\documentclass{amsart}
\usepackage{amsfonts,amssymb,amscd}
\usepackage[all]{xy}

\title[Homeomorphism and diffeomorphism groups]
{Homeomorphism and diffeomorphism groups of non-compact manifolds
 with the Whitney topology}

\author[T. Banakh]{Taras Banakh}
 \address[T. Banakh]{ Department of Mathematics,
 Ivan Franko National University of Lviv, Lviv, 79000, Ukraine, and \newline
 Instytut Matematyki,
 Uniwersytet Humanistyczno-Przyrodniczy im.\ Jana Kochanowskiego,
 Kielce, Poland}
\email{t.o.banakh@gmail.com}

\author[K. Mine]{Kotaro Mine}
 \address[K. Mine]{Institute of Mathematics,
 University of Tsukuba, Tsukuba, 305-8571, Japan}
 \email{pen@math.tsukuba.ac.jp}

\author[K. Sakai]{Katsuro Sakai$^{\ast}$}
 \address[K. Sakai]{Institute of Mathematics,
 University of Tsukuba, Tsukuba, 305-8571, Japan}
 \email{sakaiktr@sakura.cc.tsukuba.ac.jp}

\author[T. Yagasaki]{Tatsuhiko Yagasaki$^{\ast\ast}$}
\address[T. Yagasaki]{Division of Mathematics,
 Kyoto Institute of Technology, Kyoto, 606-8585, Japan}
\email{yagasaki@kit.ac.jp}

\title[Homeomorphism and diffeomorphism groups]
{Homeomorphism and diffeomorphism groups of non-compact manifolds
 with the Whitney topology}



\thanks{This work is supported by 
Grant-in-Aid for Scientific Research (No.17540061$^\ast$, No.19540078$^{\ast\ast}$).}

\newcommand{\IR}{\mathbb R}
\newcommand{\IN}{\mathbb N}

\newcommand{\w}{\omega}
\newcommand{\U}{\mathcal U}
\newcommand{\V}{\mathcal V}
\newcommand{\W}{\mathcal W}
\newcommand{\M}{\mathcal M}
\newcommand{\N}{\mathcal N}
\newcommand{\K}{\mathcal K}
\newcommand{\F}{\mathcal F}
\newcommand{\LL}{\mathcal L}

\newcommand{\supp}{\operatorname{supp}}
\newcommand{\cbox}{\boxdot}
\newcommand{\cov}{\mathrm{cov}}
\newcommand{\St}{\mathcal St}
\newcommand{\WO}{\mathsf{WO}}

\newcommand{\HH}{\mathcal H}
\newcommand{\DD}{\mathcal D}
\newcommand{\EE}{\mathcal E}

\newcommand{\im}{\operatorname{im}}

\newcommand{\id}{\mathrm{id}}

\newcommand{\cl}{\mathrm{cl}}
\newcommand{\bd}{\mathrm{bd}}
\newcommand{\tint}{\mathrm{int}}
\newcommand{\mint}{\mathrm{Int}}

\newtheorem{theo}{Theorem}
\newtheorem{prop}[theo]{Proposition}

\newtheorem*{conj}{Conjecture}

\newtheorem{theorem}{Theorem}[section]
\newtheorem{proposition}[theorem]{Proposition}

\newtheorem{lemma}[theorem]{Lemma}

\theoremstyle{definition}
\newtheorem{definition}[theorem]{Definition}

\newtheorem{remark}[theorem]{Remark}

\newtheorem{assumption}[theorem]{Assumption}
\newtheorem{convention}[theorem]{Convention}

\newcommand{\atau}[3]{\tau^{\text{\tiny$G$}}_{\text{\tiny ${#1}$}}\text{\footnotesize $({#2}, {#3})$}} 
\newcommand{\ahtau}[3]{\widehat{\tau}^{\text{\tiny$G$}}_{\text{\tiny ${#1}$}}\text{\footnotesize $({#2}, {#3})$}}

\begin{document}

\baselineskip 5.6mm

\maketitle

\begin{abstract}
For a non-compact $n$-manifold $M$
 let $\HH(M)$ be the group of homeomorphisms of $M$
 endowed with the Whitney topology
 and $\HH_c(M)$ the subgroup of $\HH(M)$ 
 consisting of homeomorphisms with compact support. 
It is shown that the group $\HH_c(M)$ is locally contractible and 
the identity component $\HH_0(M)$ of $\HH(M)$  
is an open normal subgroup in $\HH_c(M)$. 
This induces the topological factorization $\HH_c(M) \approx \HH_0(M) \times \M_c(M)$ 
for the mapping class group $\M_c(M) = \HH_c(M)/\HH_0(M)$ with the discrete topology.  
Furthermore, for any non-compact surface $M$,
 the pair $(\HH(M), \HH_c(M))$ 
 is locally homeomorphic to $(\square^\w l_2,\cbox^\w l_2)$
 at the identity $\id_M$ of $M$. 
Thus the group $\HH_c(M)$ is an $(l_2\times\IR^\infty)$-manifold.
We also study topological properties of 
 the group $\DD(M)$ of diffeomorphisms of a non-compact smooth $n$-manifold $M$
 endowed with the Whitney $C^\infty$-topology and 
 the subgroup $\DD_c(M)$ of $\DD(M)$ consisting of all diffeomorphisms 
 with compact support. 
It is shown that the pair $(\DD(M),\DD_c(M))$ is locally homeomorphic 
 to $(\square^\w l_2, \cbox^\w l_2)$ at the identity $\id_M$ of $M$.  
Hence the group $\DD_c(M)$ is a topological
 $(l_2 \times \IR^\infty)$-manifold for any dimension $n$.
\end{abstract}

\section{Introduction}

In this paper we study topological properties of groups of homeomorphisms
 and diffeomorphisms of non-compact manifolds
 endowed with the Whitney topology 
 and recognize their local topological type. 

For a $\sigma$-compact $n$-manifold $M$ possibly with boundary
 let $\HH(M)$ denote the group of homeomorphisms of $M$
 endowed with the Whitney topology 
 and $\HH_c(M)$ the subgroup of $\HH(M)$ 
 consisting of homeomorphisms with compact support.
In this topology,
 each $f \in\HH(M)$ has the fundamental neighborhood system 
$$\U(f) = \big\{g\in \HH(M):(f,g)\prec\U \big\} 
 \quad (\U\in\cov(M)),$$ 
 where $\cov(M)$ is the set of open covers of $M$ and  
 the notation $(f,g)\prec\U$ means that $f$ and $g$ are {\em $\U$-near}
 (i.e., every point $x\in M$ admits a set $U\in\U$ with $f(x), g(x) \in U$) 
(Proposition~\ref{l_W-top}). 
The {\em support} of $h \in \HH(M)$ is defined by 
$$\supp(h) = \cl\big\{x \in M : h(x) \neq x\big\}.$$
The group $\HH(M)$ is a topological group
 and the identity component $\HH_0(M)$ of $\HH(M)$ lies in
 the subgroup $\HH_c(M)$ (Proposition~\ref{compsupp}\,(1)). 

In case $M$ is a compact $n$-manifold,
 the Whitney topology on the group $\HH(M)$ coincides
 with the compact-open topology.
Hence, $\HH(M)$ is completely metrizable and locally contractible 
 (\cite{Cer}, \cite{EK}). 
When $M$ is a compact surface (or a finite graph),
 the group $\HH(M)$ is an $l_2$-manifold
 (the case of finite graphs is the result of \cite{An};
 the case of compact surfaces is a combination of
 the results of \cite{LM}, \cite{Geo} and \cite{T1}).
For $n \geq 3$, 
 it is still an open problem 
 if the homeomorphism group $\HH(M)$ of a compact $n$-manifold $M$
 is an $l_2$-manifold \cite{AB},
 which is called the Homeomorphism Group Problem \cite{W-Prob}. 

In this paper we are concerned with the case that
 $M$ is a noncompact $\sigma$-compact $n$-manifold.  
In this case the Whitney topology on $\HH(M)$ is still very important
 in Geometric Topology, 
 but it has rather bad local properties. 
Our observations in this paper mean that
 the subgroup $\HH_c(M)$ of $\HH(M)$ is quite nice
 from the topological viewpoint. 
First we recall basic facts on the local topological models
 of the groups $\HH(M)$ and $\HH_c(M)$. 

A {\em Fr\'echet space} is a completely metrizable
 locally convex topological linear space and
 an LF-space is the direct (or inductive) limit of increasing sequence of
 Fr\'echet spaces in the category of locally convex topological linear spaces.
A topological characterization of LF-spaces is given in \cite{BR}.
The simplest non-trivial example of an LF-space is $\IR^\infty$,
 the direct limit of the tower
$$\IR^1\subset \IR^2\subset \IR^3\subset\cdots,$$
 where each space $\IR^n$ is identified with
 the hyperplane $\IR^n\times\{0\}\subset\IR^{n+1}$.
In \cite{Man} P.~Mankiewicz studied LF-spaces and proved that
 an infinite-dimensional separable LF-space is homeomorphic to ($\approx$)
 either $\IR^\infty$, $l_2$ or $l_2\times\IR^\infty$. 
The space $l_2\times\IR^\infty$ is homeomorphic to
 the countable small box power $\cbox^\w l_2$ of the Hilbert space $l_2$,
 which is a subspace of the box power $\square^\w l_2$.
One should notice that
 the box power $\square^\w l_2$ is neither locally connected,
 nor sequential, nor normal (see \cite{GZ}, \cite{Wil1}).

In the papers \cite{Ban} and \cite{BMS}
 we have already shown that
$$(\HH(\IR), \HH_c(\IR)) \approx (\HH([0,\infty)), \HH_c([0,\infty)))
 \approx (\square^\w l_2, \cbox^\w l_2).$$
Moreover, it is proved in \cite{BMS} that
 if $M$ is a non-compact separable graph
 then $\HH_c(M)$ is an $(l_2\times\IR^\infty)$-manifold.
Thus, as a non-compact version of the Homeomorphism Group Problem, 
 we can expect the following:

\begin{conj}
For any non-compact $\sigma$-compact $n$-manifold $M$ possibly with boundary,
 the pair $(\HH(M), \HH_c(M))$
 is locally homeomorphic to $(\square^\w l_2,\cbox^\w l_2)$ 
 at the identity $\id_M$ of $M$. 
In particular,
 the group $\HH_c(M)$ is an $(l_2\times\IR^\infty)$-manifold.
 \footnote
{Because of the {\bf non-paracompactness} of $\square^\w l_2$,
we avoid to say  
``$\square^\w l_2$-manifold'' or
``$(\square^\w l_2,\cbox^\w l_2)$-manifold pair''.}	
\end{conj}

Here, we say that a pair $(X',X)$ of topological spaces $X\subset X'$ is {\em locally homeomorphic to} a pair $(Y',Y)$ if each point $x\in X$ has an open neighborhood $U\subset X'$ such that the pair $(U,U\cap X)$ is homeomorphic to the pair $(V,V\cap Y)$ for some open set $V\subset Y'$. 

In this paper we first show that the group $\HH_c(M)$ is locally contractible 
 in any dimension $n$. 

\begin{prop}\label{prop_Homeo_LC} 
For any $\sigma$-compact $n$-manifold $M$ possibly with boundary,
 the group $\HH_c(M)$ is locally contractible.
\end{prop}

Thus the identity component $\HH_0(M)$ is an open normal subgroup
 of $\HH_c(M)$ and it induces the topological factorization
$$\HH_c(M) \approx \HH_0(M) \times \M_c(M),$$ 
 where $\M_c(M) = \HH_c(M)/\HH_0(M)$ is the {\em mapping class group} of $M$
 with the discrete topology. 

As mentioned above, 
 the conjecture above has been proved in the case $n = 1$, see \cite{BMS}.
Here we solve the conjecture affirmatively in the case $n = 2$. 

\begin{theo}\label{thm_main_1}
Suppose $M$ is a non-compact $\sigma$-compact $2$-manifold
 possibly with boundary. 
Then the pair $(\HH(M), \HH_c(M))$ 
 is locally homeomorphic to $(\square^\w l_2,\cbox^\w l_2)$ 
 at the identity $\id_M$ of $M$. 
In particular, the group $\HH_c(M)$ is an $(l_2\times\IR^\infty)$-manifold.
\end{theo}

Any $\sigma$-compact $2$-manifold $M$ possibly with boundary
 admits a PL-structure \cite{Moise}.
When $M$ is equipped with a PL-structure, 
 the symbol $\HH^{PL}(M)$ denotes the subgroup of $\HH(M)$
 consisting of PL-homeomorphisms with respect to this  PL-structure. 
A subspace $A$ of a space $X$ is said to be {\em homotopy dense} (abbrev.\ HD)
 if there exists a homotopy $\phi_t : X \to X$ such that
 $\phi_0 = \id_X$ and $\phi_t(X) \subset A$ $(t \in (0,1])$. 

\begin{prop}\label{prop_main_2}
Suppose $M$ is a non-compact $\sigma$-compact PL 
$2$-manifold
 possibly with boundary.
Then the subgroup $\HH^{PL}_c(M)$ is homotopy dense in $\HH_c(M)$. 
\end{prop}

We also study the local topological type of
 groups of diffeomorphisms of non-compact smooth manifolds endowed
 with the Whitney $C^\infty$-topology. 
For a smooth $n$-manifold $M$,
 let $\DD(M)$ denote the group of all diffeomorphisms 
 of $M$ endowed with the Whitney $C^\infty$-topology.  
Let $\DD_0(M)$ denote the identity component of 
 the diffeomorphism group $\DD(M)$ and 
 $\DD_c(M)$ denote the subgroup of $\DD(M)$ 
 consisting of all diffeomorphisms of $M$
 with compact support. 

\begin{theo}\label{thm_main_2}
For a non-compact $\sigma$-compact smooth $n$-manifold $M$ without boundary,
 the pair $(\DD(M),\DD_c(M))$ is locally homeomorphic
 to $(\square^\w l_2, \cbox^\w l_2)$ at the identity $\id_M$ of $M$. 
In particular, 
 the group $\DD_c(M)$ is a topological $(l_2 \times \IR^\infty)$-manifold.
\end{theo}

This implies that the identity component $\DD_0(M)$
 is an open normal subgroup of $\DD_c(M)$
 and it induces the topological factorization
$$\DD_c(M) \approx \DD_0(M) \times \M_c^\infty(M),
 \hspace{5mm} \M_c^\infty(M) = \DD_c(M)/\DD_0(M).$$
 
In \cite{BY} we have shown that 
$(\DD(\IR),\DD_c(\IR)) \approx (\square^\w l_2, \cbox^\w l_2).$ 
In the succeeding paper \cite{BMSY},
 we determine the global topological types of the groups $\HH_c(M)$ 
 for non-compact surfaces $M$ and 
 the groups $\DD_c(M)$ for some kind of non-compact smooth $n$-manifolds $M$. 

This paper is organized as follows: 
Section \ref{Box-prod} contains the basic facts on
 the box products and the small box products, and 
 Section \ref{Basic} contains generalities on homeomorphism groups
 with the Whitney topology. 
In Section \ref{Trans-gr} we introduce some fundamental notations
 on transformation groups and in Section \ref{Trans-gr-strong-top}
 we formulate the notion of strong topology on transformation
 groups and study some basic properties. 
In Section \ref{Homeo-Diffeo} we apply these results
 to groups of homeomorphisms and diffeomorphisms of noncompact manifolds
 and prove Theorems~\ref{thm_main_1} and \ref{thm_main_2} together with 
Propositions~\ref{prop_Homeo_LC} and \ref{prop_main_2}. 


\section{Box and small box products} 
\label{Box-prod}

In this section
 we recall some basic properties on box products and small box products. 
Let $\w$ and $\IN$ denote the sets of non-negative integers
 and positive integers, respectively. 

\begin{definition}
(1) The {\em box product} \ $\square_{n\in\w}X_n$
 of a sequence of topological spaces $(X_n)_{n\in\w}$
is the countable product $\prod_{n\in\w}X_n$
 endowed with the box topology generated
 by the base consisting of boxes $\prod_{n\in\w}U_n$,
 where $U_n$ is an open subset of $X_n$.
\smallskip

(2) The {\em small box product} \ $\cbox_{n\in\w}X_n$ of
 a sequence of pointed spaces $(X_n,*_n)_{n\in\w}$
 is the subspace of $\square_{n\in\w}X_n$ defined by
$$\cbox_{n\in\w}X_n = \big\{(x_n)_{n\in\w}\in\square_{n\in\w}X_n :
 \exists \,m\in\w\;\forall n\ge m,\  x_n=*_n \,\big\}.$$

(3) The pair $(\square_{n\in\w}X_n, \cbox_{n\in\w}X_n)$
 is denoted by the symbol  
 $(\square, \cbox)_{n\in\w}X_n$. 
\end{definition}

The small box product $\cbox_{n\in\w}X_n$
 has a canonical distinguished point $(*_n)_{n\in\w}$. 
For a sequence of subsets $A_n \subset X_n$ ($n\in\w$),
 let
$$\cbox_{n\in\w}A_n = \cbox_{n\in\w}X_n \cap \square_{n\in\w}A_n,$$
 where it is not assumed that $\ast_n \in A_n$.
If $\ast_n \not\in A_n$ for infinitely many $n \in\w$
 then $\cbox_{n\in\w}A_n = \emptyset$.
Identifying $\prod_{i\le n}X_i$ with the closed subspace
 $\big\{(x_i)_{i\in\w}\in\cbox_{i\in\w}X_i : x_i=*_i \ (i>n) \big\}$,
 we can regard $\cbox_{n\in\w}X_n = \bigcup_{n\in\w}\prod_{i\le n}X_i$.
When $X_n = X$ for all $n\in\w$,
 we write $\square^\w X$ and $\cbox^\w X$
 instead of $\square_{n\in\w}X_n$ and $\cbox_{n\in\w}X_n$,
 which are called the {\em box power} and the {\em small box power} of $X$,
 respectively.
Then we can regard $\cbox^\w X = \bigcup_{n\in\w}X^n$,
 where $X \subset X^2 \subset X^3 \subset \cdots$.

\begin{proposition}\label{para-cbox}
If each finite product $\prod_{i\leq n} X_i$ is paracompact,
 then the small box product $\cbox_{i\in\w} X_i$ is also paracompact.
\end{proposition}

\begin{proof}
By the characterization of paracompactness (Theorem 5.1.11 in \cite{En}),
 it suffices to show that every open cover $\U$ of $\cbox_{i\in\w} X_i$
 has a $\sigma$-locally finite open refinement.
For each $n \in \IN$,
 we shall construct a locally finite open collection $\V_n$
 in $\cbox_{i\in\w} X_i$ which covers $\prod_{i\leq n} X_i$
 and refines $\U$.
Each $x \in \prod_{i\leq n} X_i$ has a basic open neighborhood
 $\cbox_{i\in\w} U_i^x$ which is contained some member of $\U$.
Then $\{\prod_{i\leq n} U_i^x : x \in \prod_{i\leq n} X_i\}$
 is an open cover of $\prod_{i\leq n} X_i$,
 which has a locally finite open refinement $\U_n$.
For each $U \in \U_n$,
 choose $x \in \prod_{i\leq n} X_i$ so that $U \subset \prod_{i\leq n} U_i^x$
 and define $V_U = U \times \cbox_{i>n} U_i^x$.
Then $\V_n = \{V_U : U \in \U_n\}$ $(n \in \w)$
 satisfy the required conditions.  
Consequently,
 $\bigcup_{n\in\IN}\V_n$ is a $\sigma$-locally finite open refinement of $\U$.
\end{proof} 

A sequence of maps $\phi^n : (X_n, \ast_n) \to (Y_n, \ast_n)$ ($n \in \w$)
 induces a continuous map 
$$\cbox_{n\in\w} \phi^n : \cbox_{n\in\w}X_n \to \cbox_{n\in\w}Y_n, \ 
(\cbox_{n\in\w} \phi^n)((x_n)_{n\in\w}) = (\phi^n(x_n))_{n\in\w}.$$  

\begin{lemma}\label{lem_map} 
For a compact space $K$ and a sequence of maps
$\phi^n :(X_n \times K, \{\ast_n\} \times K) \to (Y_n, \ast_n)$ $(n \in \w)$, 
 the map 
$$\Phi : \big(\cbox_{n\in\w}X_n\big) \times K \to \cbox_{n\in\w}Y_n, \ 
\Phi\big((x_n)_{n\in\w}, y) = (\phi^n(x_n, y))_{n\in\w}.$$ 
is continuous. 
\end{lemma}

\begin{proof} 
The proof is straightforward. 
Take any point $((x_n)_{n\in\w},y)$ of $\cbox_{n\in\w} X_n \times K$ and
 any open neighborhood $V$ of $\Phi((x_n)_{n\in\w},y)$ in $\cbox_{n\in\w} Y_n$.
We may assume that $V$ is of the form $V = \cbox_{n\in\w} V_n$, 
 where $V_n$ is an open neighborhood of $\phi^n(x_n, y)$ in $X_n$.
There exists $m \in \omega$ such that $x_n = \ast_n$ for $n > m$. 
For $n = 0, 1, \cdots, m$,
 choose open neighborhoods $U_n$ of $x_n$ in $X_n$ and $W_n$ of $y$ in $K$
 such that $\phi_n(U_n \times W_n) \subset V_n$. 
For $n > m$, since $\phi_n(\{ \ast_n \} \times K) = \{ \ast_n \} \subset V_n$ 
 and $K$ is compact,
 there exists an open neighborhood $U_n$ of $x_n = \ast_n$ in $X_n$
 such that $\phi^n(U_n \times K) \subset V_n$. 
Then $U = \cbox_{n\in\w} U_n$ and $W = \bigcap_{n=0}^m W_n$ 
 are open neighborhoods of $(x_n)_n$ in $\cbox_{n\in\w} X_n$ and $y$ in $K$,
 respectively.
Now, it is easy to see that $\Phi(U \times W) \subset V$.
This completes the proof.
\end{proof} 

For example, a sequence of homotopies  
 $\phi^n_t : (X_n, \ast_n) \to (Y_n, \ast_n)$ ($n \in \w$)
 induces a homotopy 
$$\cbox_{n\in\IN} \phi^n_t : \cbox_{n\in\IN}X_n \to \cbox_{n\in\IN}Y_n.$$  
This simple observation leads to some useful consequences. 

\begin{definition}\label{def_HD} 
A subspace $A$ of a space $X$ is called {\em homotopy dense} (abbrev.\ HD)
 in $X$ rel.\ a subset $A_0$ of $A$ if there exists a homotopy
 $\phi_t : X \to X$ $(t \in [0,1])$ such that 
 $\phi_0 = \id_X$, $\phi_t|_{A_0} = \id$ and
 $\phi_t(X) \subset A$ $(t \in (0, 1])$.  
\end{definition} 

\begin{proposition}\label{prop_HD} 
Let $(X_n, A_n, \ast_n)_{n\in\w}$ be a sequence of pointed pair of spaces. 
If each $A_n$ is HD in $X_n$ rel.\ the point $\ast_n$,
 then $\cbox_{n\in\w}A_n$ is HD in $\cbox_{n\in\w}X_n$. 
\end{proposition}

\begin{remark}\label{rem_HD} 
(1) Suppose $X$ is a metrizable space,
 $A_0 \subset A \subset X$ and $A_0$ is a closed subset of $X$. 
If $A$ is HD in $X$, then $A$ is HD in $X$ rel.\ $A_0$. 

(2) Suppose $G$ is a topological group and $H$ is a subgroup of $G$. 
If $H$ is HD in $G$, then $H$ is HD in $G$
 rel.\ the identity element $e$ of $G$. 
\end{remark}

\begin{definition}\label{d_LC}
{\rm (1)} A subspace $A$ of a space $X$ is called {\em contractible} in $X$
 (rel.\ a point $a \in A$) if there exists a homotopy
 $\phi_t : A \to X$ $(t \in [0,1])$ such that 
 $\phi_0 = \id_A$ and $\phi_1(A)$ is a singleton
 ($\phi_t(a) = a$ for every $t \in [0,1]$,
 whence $\phi_1(A) = \{a\}$). 

{\rm (2)} A space $X$ is called 
({\em strongly}) {\em locally contractible} at $x \in X$
 if every neighborhood $U$ of $x$ contains
 a neighborhood $V$ of $x$ which is contractible in $U$ (rel.\ $x$).

{\rm (3)} A pointed space $(X,\ast)$ is said to be
 (i) {\em locally contractible}
 if $X$ is locally contractible at any point of $X$ and
 strongly locally contractible at the point $\ast$, and 
 (ii) {\em contractible} if $X$ is contractible in $X$ rel.\ $\ast$. 
\end{definition}

\begin{proposition}\label{p_LC}
If pointed spaces $(X_n,\ast_n)$ $(n\in\w)$ are (locally) contractible,
 then the small box product $\cbox_{n\in\w} X_n$
 is also (locally) contractible as a pointed space.
\end{proposition}

\begin{remark}\label{loc-contr-TG}
A space $X$ is called {\em semi-locally contractible} at a point $x \in X$
if $x$ has a neighborhood $V$ in $X$ which contracts in $X$.
It is easy to see that if a topological group $G$
 is semi-locally contractible at the identity element $e \in G$
 then $G$ is strongly locally contractible at every $x \in G$,
 hence the pointed space $(G,e)$ is locally contractible.
Indeed,
 if $h : V_0 \times [0,1] \to G$ is a contraction
 of a neighborhood $V_0$ of $e \in G$, then 
we can define a contraction $h' :  V_0 \times [0,1] \to G$ by
 $h'(x,t) = h(e,t)^{-1}h(x,t)$. 
Since $h'(\{e\} \times [0,1]) = \{e\}$,
 every neighborhood $U$ of $e$ contains a neighborhood $V$ of $e$
 such that $h'(V \times [0,1]) \subset U$.
Then the restriction $h'|_{V \times [0,1]}$ is a contraction of $V$ in $U$
 fixing the identity element $e$. 
Since the topological group $G$ is homogeneous,
 it follows that $G$ is strongly locally contractible at every $x \in G$.
\end{remark}

Finally we discuss the box products of topological groups. 
As usual, we regard a topological group as a pointed space
 by distinguishing the identity element.
For topological groups $(G_n)_{n\in\w}$, 
 the box product $\square_{n\in\w}G_n$ is a topological group 
 under the coordinatewise multiplication,  
 and the small box product $\cbox_{n\in\w}G_n$ is 
 a topological subgroup of $\square_{n\in\w}G_n$.
 
Suppose $G$ is a topological group with the identity element $e \in G$. 
Any sequence $(G_n)_{n\in\w}$ of subgroups of $G$ induces 
 the natural multiplication map
$$p :\cbox_{n\in\w}G_n\to G, \ \
 p(x_0,\dots,x_k,e,e,\dots) = x_0 \cdot x_1\cdots x_k.$$

\begin{lemma}\label{lem1}
The map $p$ is continuous.
\end{lemma}

\begin{proof}
Fix any point $x=(x_0,\dots,x_k)\in\cbox_{n\in\w}G_n$
 and take any neighborhood $V$ of its image $p(x) = x_0\cdots x_k$ in $G$.
Replacing $x$ by a longer sequence if necessary,
 we can assume that $x_k = e$.
By the continuity of the group multiplication,
 find a sequence of neighborhoods $U_n$ of $x_n$ in $G$, $n\le k$, such that
$$U_0\cdot U_1\cdots U_{k-1}\cdot U_k\cdot U_k \subset V.$$

Now, inductively construct a decreasing sequence $(U_n)_{n>k}$
 of open neighborhoods of $e$ in $G$
 such that $U_n\cdot U_n \subset U_{n-1}$ for all $n>k$.
Such a choice will guarantee that
$$U_0\cdots U_{k-1} \cdot U_k \cdot U_{k+1} \cdots U_n
 \subset U_0\cdots U_{k-1}\cdot U_k \cdot U_k \subset V \quad
(n>k).$$
Then the set $U = \square_{n\in\w} U_n \cap \cbox_{n\in\w} G_n$
 is an open neighborhood of $x$ in $\cbox_{n\in\w}G_n$ and
$$p(U) \subset \bigcup_{n > k} (U_0\cdots U_n) \subset V,$$
 which proves the continuity of $p$ at $x$.
\end{proof}

Let $p : X \to Y$ be a continuous map.
A {\em local section} of $p$ at $y \in Y$ is a map $s : V \to X$
 defined on a neighborhood $V$ of $y$ in $Y$ such that $ps = \id$.
When $V = Y$, the map $s$ is called a {\em section} of $p$.

\begin{lemma}\label{lem_multi-section}
Suppose $(G_n)_{n\in\w}$ is a sequence of subgroups of $G$ such that 
$G_n \subset G_{n+1}$ for each $n \in \w$
 and $G = \bigcup_{n \in \w} G_n$. 
 If the multiplication map $p :\cbox_{n \in \w}G_n\to G$ 
has a local section at $e$,  
then the following hold: 
\begin{itemize}
\item[(1)] The map $p$ has  a local section at any point of $G$ and 
a local section $s$ at $e$ with $s(e) = (e, e, \cdots)$. 
\item[(2)] If each $G_n$ is locally contractible, then so is $G$. 
\item[(3)] Suppose $G$ is paracompact and $H$ is a subgroup of $G$. 
If $H_n = H \cap G_n$ is HD in $G_n$ for each $n \in \w$,
 then $H$ is HD in $G$. 
\end{itemize}
\end{lemma}

\begin{proof} 
(1) The verification is simple and omitted. 

(2) By Remark~\ref{loc-contr-TG} and Proposition~\ref{p_LC},
 the small box product $\cbox_{n\in\w} G_n$ is locally contractible.
Since the map $p$ has a local section at any point,
 the group $G$ is also locally contractible. 

(3) Since $G$ is paracompact,
 it suffices to show that each $g \in G$ has an open neighborhood  $U$ in $G$
 with a homotopy $\phi_t : U \to G$ such that $\phi_0 = \id$
 and $\phi_t(U) \subset H$ $(t \in (0,1])$. 
By (1) the map $p$ admits a local section $s : U \to \cbox_{n\in\w} G_n$
 at the point $g$. 
By Remark~\ref{rem_HD}\,(2) and Proposition~\ref{prop_HD},
 the small box product $\cbox_{n\in\w} H_n$ is HD in $\cbox_{n\in\w} G_n$
 by an absorbing homotopy $\psi_t$.  
Then the homotopy $\phi_t$ is defined by $\phi_t = p\psi_ts$. 
\end{proof} 


\section{Basic properties of homeomorphism groups with the Whitney topology}
\label{Basic} 

In this section, we list some basic properties of the Whitney topology 
 on homeomorphism groups. 
For any topological space $M$, 
 let $\HH(M)$ denote the group of homeomorphisms of $M$ 
 endowed with the Whitney topology.
This topology is generated by the subsets 
$$\U(h) = \big\{g\in \HH(M) : (h,g)\prec\U \big\},
 \quad (h \in \HH(M),\ \U\in\cov(M)),$$ 
 and each $h \in \HH(M)$ has the neighborhood basis
 $\U(h)$ $(\U\in\cov(M))$.
On the space $C(X,Y)$ of all continuous functions from $X$ to $Y$,
 the Whitney topology is usually defined
 as the {\em graph topology} or the {\em $\WO^0$-topology},
 that is, it is generaled by
$$\Gamma_U = \{f\in C(X,Y) : \Gamma_f \subset U\},$$
 where $U$ runs through all open sets in $X \times Y$
 and $\Gamma_f = \{(x,f(x)) : x \in C\}$ is the graph of $f \in C(X,Y)$
 (e.g., see \cite[\S41]{KM}).
The {\em graph topology} or the {\em $\WO^0$-topology} on $\HH(M)$
 is the subspace topology inherited from the space $C(M,M)$ with this topology.
In the space $C(X,Y)$,
 the graph topology is not generaled by the sets
 $$\U(f) = \{g\in C(X,Y) : (f,g)\prec\U\} \ \ 
 (f \in C(X,Y), \ \U\in\cov(Y)).$$
For completeness,
 we give a proof of the following:

\begin{proposition}\label{l_W-top}
For any topological space $M$,
 $\U(h)$ $(\U\in\cov(M))$ is a neighborhood basis
 of $h \in \HH(M)$ in the graph topology.
\end{proposition}

\begin{proof}
Fix $h \in \HH(M)$. (1) Let $W \subset M^2$ be an open set
 such that $\Gamma_h \subset W$.
Each $x \in M$ has an open neighborhood $U_x$ in $M$
 such that $U_x \times h(U_x) \subset W$.
Since $h$ is a homeomorphism,
 $\U= \{h(U_x) : x \in M\} \in \cov(M)$.
To see $\U(h) \subset \Gamma_W$,
 take any homeomorphism $g \in \U(h)$.
For every point $z \in M$, there exists $y\in M$ such that
 $\{h(z),g(z)\}\subset h(U_y)$.
Since $h$ is a bijection, we have $z\in U_y$, hence
 $(z,g(z)) \in U_y \times h(U_y) \subset W$.
This means that $\Gamma_g \subset W$ and hence $\U(h) \subset \Gamma_W$.

(2) Take any cover $\U \in \cov(M)$.
For each $x \in M$,
 choose $U_x \in \U$ with $h(x)\in U_x$. Then  
$$W=\bigcup_{x\in M} h^{-1}(U_x) \times U_x \subset M \times M$$
 is an open neighborhood of $\Gamma_h$ in $M \times M$. 
To see $\Gamma_W \subset \U(h)$, take any $g \in \Gamma_W$.
Since $\Gamma_g \subset W$, for any $y\in M$ 
we can find $x\in M$ such that
 $(y,g(y)) \in h^{-1}(U_x) \times U_x$ and hence $\{h(y),g(y)\}\subset U_x$.
This means that $g \in \U(h)$. 
\end{proof}

For subspaces $K \subset L \subset M$, 
 let $\EE_K(L,M)$ denote the space of embeddings $f : L \to M$
 with $f|_K = \id_K$ endowed with the compact-open topology. 
In comparison with the Whitney topology (when $M$ is Hausdorff),
 each $f \in \EE_K(L,M)$ admits the fundamental neighborhood system: 
\[ 
{\mathcal U}(f, C) = \{ g \in \EE_K(L,M) :
 (f|_C, g|_C) \prec \U \} 
\hspace{3mm} \mbox{($C$ is a compact subset of $L$,\ ${\mathcal U} \in \cov(M)$).} \]  
The group $\HH(M)$ acts on $\EE(L,M)$ by the left composition. 
When $M$ is paracompact,
 every open cover of $M$ admits a star-refinement.
This remark leads to the following basic fact.

\begin{proposition}\label{W-top-group} 
If $M$ is paracompact,\footnote
	{It is proved in \cite{Gau}	that
	$\HH(X)$ is a topological group if $X$ is metrizable.}
 then {\rm (i)} $\HH(M)$ is a topological group and 
{\rm (ii)} the natural action of $\HH(M)$ on $\EE(L,M)$ is continuous. 
\end{proposition} 

\begin{proof}
For the sake of completeness we include the proof.

(i) It follows from the definition of the Whitney topology on $\HH(M)$
 that the inversion $f\mapsto f^{-1}$ is continuous.
So, it remains to check that
 the composition is continuous with respect to the Whitney topology.
Given 
$f,g\in\HH(M)$ and $\U \in \cov(M)$, 
 we should find $\V, \W \in \cov(M)$ such that
 $f'g' \in \U(fg)$ for every $f' \in \V(f)$ and $g' \in \W(g)$,
 that is, $(f,f') \prec \V$ and $(g,g') \prec \W$ imply $(fg,f'g') \prec \U$.
By the paracompactness of $M$,
 there is a cover $\V\in\cov(M)$ with $\St(\V)\prec\U$.
Let $\W=f^{-1}(\V)=\{f^{-1}(V):V\in\V\}$ and
 assume $(f,f') \prec \V$ and $(g,g') \prec \W$.
Since $(fg,fg') \prec f(\W) = \V$ and $(fg',f'g') \prec \V$,
 it follows that $(fg,f'g') \prec \St(\V) \prec \U$.

The assertion (ii) can be seen by the same argument. 
\end{proof}

The next proposition is the main result in this section. 
For a subset $L \subset M$, let $\HH(M, L) = \{ h \in \HH(M) : h|_L = \id \}$.

\begin{proposition}\label{compsupp}
If $M$ is paracompact then {\rm (1)} $\HH_0(M) \subset \HH_c(M)$ and 
{\rm (2)} every compact subspace $\K\subset \HH_c(M)$
 is contained in $\HH(M, M \setminus K)$
 for some compact subset $K \subset M$.
\end{proposition}

\begin{proof}
(1) It suffices to show that each $f \in \HH(M) \setminus \HH_c(M)$
 can be separated from $\id_M$ by a clopen subset $\U$ of $\HH(M)$.
For this purpose, we compare the space $\HH(M)$
 with the additive group $\square^\omega \IR$.
The latter space contains the clopen subgroup
$$c_0 = \left\{(a_n)_{n\in\w} \in \square^\omega \IR 
: \lim_{n\to\infty} a_n = 0 \right\}.$$

For any $f \in \HH(M) \setminus \HH_c(M)$,
 we can find a countable discrete subset $X = \{x_n\}_{n\in\w}$ of $M$
 such that $f(X) \cap X = \emptyset$.
Indeed, the set $F=\supp(f)$ is non-compact and closed in $M$,
 whence $F$ is paracompact. 
Since the compactness coincides with the pseudocompactness
 in the class of paracompact spaces (see \cite[3.10.21, 5.1.20]{En}),
 $F$ is not pseudocompact and hence admits a continuous unbounded function
 which extends to a continuous function $\xi:M\to[0,+\infty)$
 by the normality of $M$.
Since $\xi|_F$ is unbounded,
 we can choose a countable subset $X = \{x_n\}_{n\in\w}$ in $F$
 so that for each $n \in \w$, \ $f(x_n) \not= x_n, \ \xi(x_n) > n$ \ and 
$$\xi(x_n) >\max\big\{\xi(x_i),\ \xi(f(x_i)),\ \xi(f^{-1}(x_i)) : i < n \big\}.$$ 
Then $f(X) \cap X = \emptyset$ and $\lim_{n\to\infty}\xi(x_n)=\infty$.

By the normality of $M$,
 there exists a Urysohn map $\lambda : M \to [0,1]$ 
 with $\lambda(X) = 0$ and $\lambda(f(X)) = 1$.
Since $h(X)$ is discrete for each $h \in \HH(M)$,
 we have the map $\varphi : \HH(M) \to \square^\omega \IR$
 defined by $\varphi(h) = \big( \lambda(h(x_n)) \big)_{n\in\w}$.  
Since $\varphi(\id_M) = (0,0,\dots)$ and $\varphi(f) = (1,1,\dots)$, 
it follows that $\U \equiv \varphi^{-1}(c_0)$
 is a clopen neighborhood of $\id_M$ with $f \not\in \U$.
\smallskip

(2) Given a compact subset $\K\subset\HH_c(M)$ we shall show that
 $\supp(\K) = \cl_M \big(\bigcup_{h\in\K}\supp(h)\big)$ is compact. 
Assume conversely that the set $\supp(\K)$ is not compact. 
The same argument as in (i) yields a continuous function $\xi:M\to[0,\infty)$
 whose restriction $\xi|\supp(\K)$ is unbounded.
By induction we can choose sequences of points $x_n \in\supp(\K)$
 and of homeomorphisms $h_n \in \K$  ($n \in \w$) such that \ 
$h_n(x_n) \neq x_n$, \ $\xi(x_n) > n$ \ and 
$$\xi(x_n) > \max \xi\big(\textstyle \bigcup_{i<n}\supp(h_i)\big).$$
It is seen that 
 $x_n \in \supp(h_n) \setminus \bigcup_{i<n}\supp(h_i)$
 and so $x_n \neq x_m$ if $n \neq m$. 

By the compactness of $\K$,
 the sequence $(h_i)_{i\in\IN}$ has a cluster point
 $h_\infty \in \K \subset \HH_c(M)$.
Note that $U_0 = M \setminus \{x_n\}_{n\in\w}$ is open in $M$.
For each $n \in \w$,
 take a small open neighborhood $U_n$ of $x_n$ in $M$
 such that $h_n(x_n) \not\in U_n$ and $U_n \cap U_m = \emptyset$ if $n \neq m$.
Then we have $\U = \{ U_n \}_{n\in\w} \in \cov(M)$.
The neighborhood $\U(h_\infty)$ of $h_\infty$
 contains only finitely many $h_n$ because
 $h_\infty(x_n) = x_n \in U_n$ for sufficiently large $n \in \w$,
 but $h_n(x_n) \not\in U_n$ for each $n \in \w$.
This contradicts the choice of $h_\infty$
 as a cluster point of $(h_n)_{n\in\w}$.
\end{proof}

Considering Conjecture in Introduction,
 we are concerned with the paracompactness of the space $\HH_c(M)$.

\begin{proposition}\label{para-H_c}
For a locally compact separable metrizable space $M$,
 the space $\HH_c(M)$ is (strongly) paracompact.
\end{proposition}

\begin{proof}
We can write $M = \bigcup_{n\in\IN} M_n$,
 where each $M_n$ is compact and $M_n \subset \tint_M M_{n+1}$.
Then $\HH_c(M) = \bigcup_{n\in\IN} \HH(M,M \setminus \tint M_n)$.
For each $n \in \IN$,
 $\HH(M,M \setminus \tint M_n) \approx \HH(M_n,\bd_M M_n)$
 is separable metrizable, hence Lindel{\"o}f.
Therefore, $\HH_c(M)$ is Lindel{\"o}f by Theorem 3.8.5 in \cite{En},
 so it is (strongly) paracompact by Theorem 5.1.2 (Corollary 5.3.11)
 in \cite{En}.
\end{proof}

This proposition has a further refinement.
Following 
E.~Michael \cite{Mi} (see also \cite{Gru})
 we define a regular topological space $X$ to be an {\em $\aleph_0$-space}
 if $X$ possesses a countable family $\mathcal N$ of subsets of $X$
 such that for each open subset $U$ of $X$ and a compact subset $K$ in $U$
 there is a finite subfamily $\mathcal F$ of $\mathcal N$
 with $K\subset \bigcup\mathcal F\subset U$.
Such a family $\mathcal N$ is called a {\em $k$-network} for $X$.
It is clear that each $\aleph_0$-space has countable network weight
 (cf.\ \cite[p.127]{En})
 and hence is Lindel\"of \cite[Theorem 3.8.12]{En}.

\begin{proposition}\label{prop_aleph_0}
For a locally compact separable metrizable space $M$,
 the space $\HH_c(M)$ is an $\aleph_0$-space.
\end{proposition}

\begin{proof}
In the proof of Proposition~\ref{para-H_c},
 each $\HH(M,M \setminus \tint M_n)$  is separable metrizable
 and hence its topology has countable base $\mathcal B_n$.
We claim that the union $\mathcal B = \bigcup_{n\in\IN}\mathcal B_n$
 is a countable $k$-network for $\HH_c(M)$.
Indeed, given an open set $U\subset\HH_c(M)$ and a compact subset $K\subset U$,
 we can apply Proposition~\ref{compsupp}\,(2) in order to find $n\in\IN$
 such that $K\subset \HH(M,M\setminus\tint M_n)$.
The compactness of $K$ allows us to find a finite subfamily $\F$
 of the base $\mathcal B_n$ such that  $K\subset \bigcup\F\subset U$.
\end{proof}


\section{Transformation groups}
\label{Trans-gr}

\subsection{Generalities on transformation groups}
\label{Gen-trans-gr} \mbox{}

Throughout the article, a {\em transformation group} $G$ on a space $M$
 means a topological group $G$ acting on $M$ continuously and effectively.
Each $g \in G$ induces a homeomorphism of $M$,
 which is also denoted by the same symbol $g$.
This determines the canonical injection $G \hookrightarrow \HH(M)$.
Let $G_0$ denote the connected component of the identity element $e$ in $G$
 and let $G_c = \{g \in G : \supp(g) \text{ is compact}\}$. 
For any subsets $K, N$ of $M$, 
we obtain the following subgroups of $G$:
$$G_K = \{g \in G : g|_K = \id_K \}, \quad G(N) = G_{M \setminus N}, \quad 
G_K(N) = G_K \cap G(N), \quad G_{K, c} = G_K \cap G_c.$$  

\begin{proposition}\label{prop_tr-gr_compsupp} 
Suppose the natural injection $G \to \HH(M)$ is continuous.
\begin{itemize}
\item[(1)]
 If $M$ is paracompact, then $G_0 \subset G_c$ and
 every compact subspace $\K \subset G_c$ is contained in $G(K)$
 for some compact subset $K \subset M$.
\item[(2)]
 If $M$ is locally compact and $\sigma$-compact and
 $G(K)$ is second countable for every compact subset $K \subset M$,
 then $G_c$ is an $\aleph_0$-space \textup{(}hence \textup{(}strongly\textup{)} paracompact\textup{)}. 
\end{itemize}
\end{proposition}

\begin{proof}
The statement (1) follows immediately from Proposition~\ref{compsupp}, 
 while (2) can be deduced from (1)
 by analogy of the proof of Proposition~\ref{prop_aleph_0}.
\end{proof}

For any subgroup $H$ of $G$,
 there is a natural projection 
 $\pi : G \to G/H$.
The coset space $G/H=\{gH:g\in G\}$ is endowed with the quotient topology. 
The left coset $\pi(g) = gH \in G/H$ is also denoted by $\overline{g}$.
The symbols $H_0$ and $(G/H)_0$ denote the connected components  
of $e$ in $H$ and $\overline{e}$ in $G/H$ respectively. 

For each subset $L \subset M$,
 let $\EE^G(L,M)$ be the set of embeddings $g|_L : L \to M$
 induced by $g \in G$. 
The inclusion map $i_L : L \subset M$ is regarded
 as the distinguished point of this set. 
The group $G$ acts transitively on the set $\EE^G(L,M)$ by $g \cdot h = gh$. 
More generally, for subsets $K \subset L \subset N$ of $M$, 
we have the subset 
$$\EE_K^G(L, N) = \{ g|_L \in \EE(L,M) : g \in G_K(N) \} 
 =  G_K(N) \cdot i_L \subset \EE^G(L, M).$$ 
The subgroup $G_K(N)$ acts on this subset transitively
 and the subgroup $G_L(N)$ is the stabilizer of $i_L$ under this action.
Since for each $g, g' \in G_K(N)$,
$$g|L = g'|L \ \Longleftrightarrow \ g^{-1}g' \in G_L(N)
\  \Longleftrightarrow \ gG_L(N) = g'G_L(N),$$ 
 the restriction map $r : G_K(N) \to \EE_K^G(L, N)$, $r(g) = g|_L$, 
has the following factorization:
$$\xymatrix@M+1pt{
& G_K(N) \ar[dl]_{\mbox{$\pi$}} \ar[dr]^{\mbox{$r$}} & \\
G_K(N)/G_L(N) \ar[rr]^{\mbox{\ \ $\phi$}}_{\ \ \text{bij.}} & & \EE_K^G(L, N), 
}$$
where $\phi(\overline{g}) = g|_L = g \cdot i_L$. 
\noindent
The map $\phi$ is a $G_K(N)$-equivariant bijection. 
For $K \subset L_2 \subset L_1 \subset N$, we obtain 
the restriction map $\EE_K^G(L_1, N) \to \EE_K^G(L_2, N)$. 
Hereafter, in case $K = \emptyset$,
 the symbol $K$ is omitted from the notations. 

Here we include general remarks on topologies on the set $\EE^G_K(L, N)$. 

\begin{definition}\label{def_adm-top} 
A topology $\tau$ on the set $\EE^G_K(L, N)$ is said to be {\em admissible} 
 if the action of $G_K(N)$ on $(\EE_K^G(L, N), \tau)$ is continuous.
 The space $(\EE_K^G(L, N), \tau)$ is also denoted by $\EE_K^G(L, N)^\tau$. 
\end{definition}

Let $\widehat{\tau} = \ahtau{K}{L}{N}$ denote 
the quotient topology on $\EE^G_K(L, N)$ induced by the map $r$.
For simplicity, the space $(\EE^G_K(L, N), \widehat{\tau})$ is denoted by 
$\widehat{\EE}^G_K(L, N)$. 

\begin{remark}\label{rem_quotient-top} 
(i) The quotient topology $\widehat{\tau}$ is the strongest admissible topology on $\EE_K^G(L, N)$. 

(ii) The restriction map
 $\widehat{\EE}^G_{K_1}(L_1,N_1) 
 \to \widehat{\EE}^G_{K_2}(L_2,N_2)$
 is continuous for triples $(N_1, L_1, K_1)$ and $(N_2, L_2, K_2)$
 such that $N_1 \subset N_2$, $L_1 \supset L_2$ and $K_1 \supset K_2$.
\end{remark}

\begin{remark}\label{rem_adm-top} 
If the set $\EE^G_K(L, N)$ is equipped with an admissible topology $\tau$,  
then the following hold: 
\begin{itemize} 
\item[(i)\ ] The maps $r$ and $\phi$ are continuous. 
The map $\phi$ is a homeomorphism if and only if $\tau = \widehat{\tau}$. 

\item[(ii)\,] The map $r : G_K(N) \to \EE^G_K(L, N)^\tau$
 has a local section at $i_L$ if and only if
 $\tau = \widehat{\tau}$ and 
the map $$\pi : G_K(N) \to G_K(N)/G_L(N)$$ has a local section. 
In this case, the map $r$ is a principal $G_L(N)$-bundle. 
\end{itemize} 
\end{remark}
 

\subsection{Local section property}\label{LSP} \mbox{}

Throughout this subsection, we assume that  
$G$ is a transformation group on a space $M$.
Suppose $K \subset L \subset N$ are subsets of $M$. 

\begin{definition}\label{d_LSP} 
We say that the triple $(N, L, K)$ has
 the {\em local section property} for $G$ (abbrev.\ LSP$_G$)
 if the restriction map 
 $r : G_K(N) \to (\EE^G_K(L, M), \tau)$ has
 a local section $s$ at the inclusion $i_L : L \subset M$
 with respect to some admissible topology $\tau$.
(In this definition,
 one should not confuse $\EE^G_K(L, M)$ with $\EE^G_K(L, N)$.)
\end{definition} 

\begin{remark}\label{rem_LSP_G} {\rm (0)} 
The above map $s$ is also a local section of the restriction map $G_K \to \EE^G_K(L, M)^\tau$. 
Thus $\tau = \ahtau{K}{L}{M}$ by Remark~\ref{rem_adm-top}~\!(ii).
\smallskip 

(i) We can modify the local section $s$ so that $s(i_{L}) = \id_M$. 
\smallskip 

(ii) A triple $(N, L, K)$ has LSP$_G$ if and only if 
 for some admissible topology $\tau$
\begin{itemize} 
\item[(a)] $\EE^G_K(L,N)$ is open in
 $\EE^G_K(L,M)^\tau$ and 
\item[(b)] the map $r : G_K(N) \to \EE^G_K(L,N)^\tau$
 has a local section at $i_L$.
\end{itemize} 
\smallskip 

(iii) Suppose $(N_1, L, K_1)$ and $(N_2, L, K_2)$ are triples
 of subsets of $M$ such that $N_1 \subset N_2$ and $K_1 \subset K_2$. 
If $(N_1, L, K_1)$ has LSP$_G$, then so does $(N_2, L, K_2)$.
Indeed, the map 
$r_1 : G_{K_1}(N_1) \to \EE^G_{K_1}(L,M)^\tau$
 has a local section  
 $s : \U \to G_{K_1}(N_1)$
 at the inclusion $i_L : L \subset M$.
For each $f \in \U \cap \EE^G_{K_2}(L,M)^\tau$,
 since $s(f)|_L = f$ and $K_2 \subset L$,
 we have $s(f)|_{K_2} = f|_{K_2} = \id$,
 hence the restriction of $s$ is a local section
 for $r_2 : G_{K_2}(N_2) \to \EE^G_{K_2}(L,M)^\tau$.
\end{remark}

\begin{definition} \label{d_WLSP} 
Suppose $\tau$ is an admissible topology on $\EE^G_K(N, M)$. 
We say that the triple $(N, L, K)$ has the {\em weak local section property}
 for $G$ with respect to $\tau$ (abbrev.\ WLSP$_{G, \tau}$) 
 if there exists an open neighborhood $\V$ of $i_N$ in $(\EE^G_K(N, M),\tau)$
 and a continuous map $s : \V \to G_K(N)$ such that
 $s(f)|_L = f|_L$ for each $f \in \V$. 
\end{definition}

\begin{remark}\label{rem_WLSP}
(i) We can modify the map $s$ so that $s(i_{N}) = \id_M$. 
\begin{itemize}
\item[(ii)] 
When the topology $\tau$ is understood from the context,
 we omit the symbol $\tau$ from the notations. 
\end{itemize} 
\end{remark} 
 

\subsection{Exhausting sequences}\label{E-seq} \mbox{}

This subsection includes a remark on exhausting sequences of spaces 
and their associated towers in transformation groups. 
Suppose $M$ is a locally compact $\sigma$-compact space and 
$G$ is a transformation group on $M$. 
Recall that a subset $A$ of $M$ is {\em regular closed}
 if $A = \cl_M (\tint_M A)$ and that 
 a sequence $\F = (F_i)_{i\in\IN}$ of subsets of $M$
 is {\em discrete} if each point $x\in M$ has a neighborhood
 which meets at most one $F_i$. 

By the assumption,
 there exists a sequence $(M_i)_{i \in \IN}$
 of compact regular closed sets in $M$ 
 such that $M_ i \subset \tint_M M_{i+1}$ 
 and 
 $M = \bigcup_{i\in\IN}M_i$ ($= \bigcup_{i\in\IN}\tint_M M_i$).
It induces the tower $(G(M_i))_{i\in\IN}$ of closed subgroups of $G_c$
 and the multiplication map
$$p : \cbox_{i\in\IN} G(M_i) \to G_c,
 \quad p(h_1, \dots, h_n) = h_1 \cdots h_n.$$
For each $i \in \IN$,
 let $K_i = M \setminus \tint_M M_i$ and
 $L_i = M_i \setminus \tint_M M_{i-1}$, where $M_0=\emptyset$.
Then the sequences $(L_{2i-1})_{i\in\IN}$ and $(L_{2i})_{i\in\IN}$
 are discrete in $M$. 
There exists a sequence $(N_i)_{i\in\IN}$ of compact subsets of $M$
 such that $L_i \subset \tint_M N_i$
 and $N_i \cap N_j = \emptyset$ if $|i - j| \geq 2$
 (hence $N_i \subset \tint_M M_{i+1} \setminus M_{i-2}$ and
 subsequences $(N_{2i-1})_{i\in\IN}$ and $(N_{2i})_{i\in\IN}$
 are discrete in $M$).
Note that $G(M_i) = G_{K_i}$ and $L_i$ is regular closed since  
$L_i = \cl_M((\tint_M M_i) \setminus M_{i-1}))$. 
We call each of the sequences
 $(M_i)_{i\in\IN}$, $(M_i, L_i, N_i)_{i\in\IN}$ and
 $(M_i, K_i, L_i, N_i)_{i\in\IN}$ an {\em exhausting sequence} for $M$.
 
The next lemma directly follows from Lemma~\ref{lem_multi-section}. 

\begin{lemma}\label{lem_loc-sec_exh}
Suppose $(M_i)_{i\in\IN}$ is an exhausting sequence for $M$ and 
 the map $p : \cbox_{i\in\IN} G(M_i) \to G_c$ 
 has a local section at the identity element $e$. 
Then the following hold:
\begin{itemize}
\item[(1)] The map $p$ has a local section at any point of $G_c$. 
\item[(2)] If each $G(M_i)$ is locally contractible, then so is $G_c$.  
\item[(3)] If $G_c$ is paracompact, $H$ is a subgroup of $G$
 and each $H(M_i)$ is HD in $G(M_i)$,
 then $H_c$ is HD in $G_c$. 
\end{itemize}
\end{lemma}


\section{Transformation groups with strong topology}
\label{Trans-gr-strong-top} 

In Section \ref{Homeo-Diffeo},
 we shall study the topological properties of
 the homeomorphism group $\HH(M)$ and the diffeomorphism group $\DD(M)$
 of a non-compact manifold $M$ endowed with the Whitney topology.
These groups acts naturally on the manifold $M$ and
 admit infinite products of elements with discrete supports.
This geometric property distinguishes these groups
 from other abstract topological groups
 and connects them to the box topology.
To clarify this situation,
 we introduce the notion of transformation groups with strong topology
 and study its basic properties.
In particular, we obtain some conditions under which 
 transformation groups are locally homeomorphic to some box/small box products
 (Propositions~\ref{p_loc-homeo} and \ref{prop_loc-sec}). 
In Section \ref{Homeo-Diffeo},
 these fundamental results are applied to yield main results of this article.
Our axiomatic approach is also intended for further applications
 to the study of other subgroups of $\HH(M)$ and $\DD(M)$.

\subsection{Transformation groups with strong topology} \mbox{} 

Throughout this subsection, 
 let $G$ be a transformation group on a space $M$. 
For each discrete sequence $\LL = (L_i)_{i\in\IN}$ of subsets of $M$,
 we have 
\begin{itemize}
\item[(i)\,] the group homomorphism \ 
$\lambda_\LL : \square_{i\in\IN}G(L_i) \to \HH(M)$ \ defined by 
$$\lambda_\LL((g_i)_{i\in\IN})|_{L_j} = g_j|_{L_j} \ \ (j \in \IN) \  
\ \ \ \text{and} \ \ \ \ 
\lambda_\LL((g_i)_{i\in\IN})|_{M \setminus\bigcup_{i\in\IN} L_i} = \id.$$ 
\item[(ii)] the function \ 
 $r_\LL : G \to \square_{i\in\IN} \EE^G(L_i, M)$ \ 
 defined by 
 $$r_\LL(g) = (g|_{L_i})_{i\in\IN}.$$  
\end{itemize}
Note that $\lambda_\LL^{-1}(\HH_c(M)) = \cbox_{i\in\IN} G(L_i)$
 if each $L_i$ is compact.

\begin{definition} \label{d_strong-top}
We say that the transformation group $G$ on $M$ has {\em a strong topology}  
 if it satisfies the following conditions:
\begin{itemize}
\item[(1)]
The natural injection $G \to \HH(M)$ is continuous
 with respect to the Whitney topology on $\HH(M)$.
\item[(2)]
For any discrete sequence $\LL = (L_i)_{i\in\IN}$ in $M$,  
$$\mbox{(i) $\im \lambda_\LL \subset G$ \quad and \quad (ii) 
 the map $\lambda_\LL : \square_{i\in\IN}G(L_i)
 \to G\big(\bigcup_{i\in\IN} L_i\big)$ \ is an open embedding.}$$ 
\end{itemize}
\end{definition}

\begin{definition} \label{d_strong-top-1}
Let $\F$ be a collection of subsets of $M$. 
We say that the transformation group $G$ on $M$ has
 {\em a strong topology with respect to} $\F$ 
 if $G$ has a strong topology and satisfies the following additional condition:
\begin{itemize}
\item[$(\ast)$]
For any discrete sequence $\LL = (L_i)_{i\in\IN}$ in $M$ with $L_i \in \F$,
 the function $r_\LL : G \to \square_{i\in\IN} \widehat{\EE}^G(L_i, M)$
 is continuous. 
\end{itemize}
\end{definition}

\subsection{Transformation groups locally homeomorphic to box products}
\label{Trans-gr-box} \mbox{}

\begin{assumption}
Throughout this subsection,
 we assume that $M$ is a locally compact $\sigma$-compact space
 and $G$ is a transformation group on $M$ with 
 a strong topology with respect to a collection $\F$ of subsets of $M$. 
\end{assumption}

For notational simplicity, 
 for pairs of spaces and maps between them,
 we use the following terminology and notations:
For pairs of spaces $(X, A)$ and $(Y, B)$,
 we set $(X, A)\times (Y, B) = (X \times Y, A \times B)$.
We say that
\begin{itemize} 
\item[(i)\,]
 $(X, A)$ and $(Y, B)$ are {\em locally homeomorphic}
 and write 
 $(X, A) \approx_\ell (Y, B)$
 if for each point $a \in A$
 there exists an open neighborhood $U$ of $a$ in $X$
 and an open subset $V$ of $Y$
 which admit a homeomorphism of pairs of spaces
 $(U, U \cap A) \approx (V, V \cap B)$.
\item[(ii)]
 a map $p : (X, A) \to (Y, B)$ of pairs of spaces
 has a local section at a point $b \in B$, 
 if there exists an open neighborhood $V$ of $b$ in $Y$ and
 a map $s : (V, V \cap B) \to (X, A)$ of pairs such that $ps = \id_V$.
\end{itemize}

Suppose $\LL = (L_i)_{i\in\IN}$, $\N = (N_i)_{i\in\IN}$ and
 $\K = (K_i)_{i\in\IN}$ are discrete sequences of compact subsets of $M$ 
 such that $L_i \subset N_i$ for each $i \in \IN$ and $M = L \cup K$,
 where $L = \bigcup_{i\in\IN} L_i$ and $K = \bigcup_{i\in\IN} K_i$. 
Since $G$ has a strong topology, 
 the discrete sequences $\N$ and $\K$ induce the open embeddings 
$$\lambda_\N : \square_{i\in\IN} G(N_i) \to G(N) \quad\text{and}\quad 
 \lambda_\K: \square_{i\in\IN} G(K_i) \to G(K),$$ 
 where $N = \bigcup_{i\in\IN} N_i$. 
Then we obtain the map 
$$\rho : (\square,\cbox)_{i\in\IN} G(N_i) 
\times (\square,\cbox)_{i\in\IN} G_L(K_i) \to (G, G_c),$$  
defined by \ \ 
$\rho((g_i)_{i\in\IN},(h_i)_{i\in\IN}) 
= \lambda_\N((g_i)_{i\in\IN})\lambda_\K((h_i)_{i\in\IN}).$ 

\begin{lemma}\label{l_LSP}
 If $(N_i, L_i)$ has LSP$_G$ and $L_i \in \F$ for each $i \in \IN$, then
\begin{itemize}
\item[(1)]
 $(G, G_c) \approx_\ell (\square,\cbox)_{i\in\IN} \widehat{\EE}^G(L_i,M)
 \times (\square,\cbox)_{i\in\IN}G_L(K_i),$
\item[(2)]
 the map $\rho$ has a local section at $\id_M$. 
\end{itemize}
\end{lemma}

\begin{proof}
We are concerned with the following maps:
$$\begin{array}[c]{c}
\begin{array}[b]{ccccccl}
& r_\LL & & s & & \lambda_\N & \\[-1mm]
\V & \longrightarrow  & \square_{i\in\IN} \V_i &
 \longrightarrow & \square_{i\in\IN} G (N_i) & \longrightarrow & G(N),
\end{array} \\[3mm]
\begin{array}[b]{ccccccc}
& \phi & & s \times \id & & \theta \\[-1mm]
\V & \longrightarrow  & \square_{i\in\IN} \V_i \times G_L &
 \longrightarrow & \square_{i\in\IN} G(N_i) \times G_L & \longrightarrow & G,
\end{array} \\[4mm] 
\eta = \lambda_\N s \,r_\LL, \ \ \theta (s \times \id) \phi = \id_\V,
 \hspace{3mm} \psi = \theta (s \times \id).
\end{array}$$ 
\vskip 2mm \noindent
These are defined as follows: 
Since $G$ has a strong topology with respect to $\F$ 
and $L_i \in \F$ ($i \in \IN$),
 we obtain the map
$$r_\LL : G \to \square_{i\in\IN} \widehat{\EE}^G(L_i,M),\quad
 r_\LL(g) = (g|_{L_i})_{i\in\IN}.$$
By the assumption, for each $i \in \IN$,
 there exists an open neighborhood $\V_i$ of
 the inclusion $i_{L_i} : L_i \subset M$ in $\widehat{\EE}^G(L_i, M)$
 and a map $s_i : \V_i \to G(N_i)$ such that
 $s_i(f)|_{L_i} = f|_{L_i}$ for each $f \in \V_i$ and $s_i(i_{L_i}) = \id_M$.
The maps $s_i$ ($i \in \IN$) determine the map
$$s : \square_{i\in\IN} \V_i \to \square_{i\in\IN} G(N_i),
 \quad s((f_i)_{i\in\IN}) = (s_i(f_i))_{i\in\IN}.$$ 
The preimage $\V = r_\LL^{-1}(\square_{i\in\IN} \V_i)$
 is an open neighborhood of $\id_M$ in $G$.
Let $\eta = \lambda_\N s \,r_\LL : \V \to G(N)$.
For every $g \in \V$,
 since $\eta(g)|_{L_i} = g|_{L_i}$ for each $i \in \IN$,
 we have $\eta(g)^{-1} g \in G_L$.
The maps $\phi$ and $\theta$ are defined by
$$\phi(g) = (r_\LL(g), \eta(g)^{-1} g) \quad \text{and} \quad  
\theta((g_i)_{i\in\IN}, h) = \lambda_{\N}((g_i)_{i\in\IN})h.$$
Then $\theta (s \times \id) \phi = \id_\V$ because
 $\theta(s \times \id)\phi(g) = \lambda_\N sr_\LL(g) \eta(g)^{-1} g = g$
 for each $g \in \V$. 
Since $M = L \cup K$,
 we have $G_L \subset G(K)$.
It follows from the definition of $\lambda_\K$ that 
$$(\square, \cbox)_{i\in\IN} G_L(K_i) \begin{array}[b]{c}\lambda_{\K} \\ 
\approx 
\end{array} 
 (\im \lambda_\K \cap G_L, \im\lambda_\K \cap G_{L,c})
 \approx_\ell (G_L,G_{L,c}).$$ 

(1) Consider the map $\psi = \theta(s \times \id)$.
For each $((f_i)_{i\in\IN},h) \in \square_{i\in\IN} \V_i \times G_L$,
 we can write
$$\psi((f_i)_{i\in\IN},h) = fh, \quad\text{where}\quad 
f = \lambda_{\N}((s_i(f_i))_{i\in\IN}).$$
Since $fh|_{L_i} = s_i(f_i)|_{L_i} = f_i \in \V_i$,
 it follows that 
 $$r_\LL(fh) = (f_i)_{i\in\IN} \in \square_{i\in\IN}\V_i,$$
 which means $fh \in \V$.
Thus, we obtain the map $$\psi : \square_{i\in\IN} \V_i \times G_L \to \V$$ 
 such that $\psi\phi = \theta(s \times \id)\phi = \id_\V$.
Moreover,
$$\eta(fh)^{-1}fh = \lambda_\N((s(f_i))_{i\in\IN})^{-1}fh = h.$$
Since $r_\LL(fh) = (f_i)_{i\in\IN}$,
 it follows that $\phi\psi((f_i)_{i\in\IN},h) = ((f_i)_i, h)$.
Therefore, $\phi$ is a homeomorphism with $\phi^{-1} = \psi$.
On the other hand,
 $\phi(\V \cap G_c) = \cbox_{i\in\IN} \V_i \times G_{L,c}$
 because $s_i(i_{L_i}) = \id_M$.
Consequently, 
$$(\V, \V \cap G_c) \approx
 (\square,\cbox)_{i\in\IN} \V_i \times (G_L, G_{L,c}).$$ 
Recall $(G_L, G_{L,c}) \approx_\ell (\square, \cbox)_{i\in\IN} G_L(K_i)$.
Thus, we have 
$$(G, G_c) \approx_\ell (\square, \cbox)_{i\in\IN} \widehat{\EE}^G(L_i,M)
 \times (\square, \cbox)_{i\in\IN} G_L(K_i).$$

(2) The map $\rho$ has the factorization 
$$\xymatrix@M+4pt{
\square_{i\in\IN} G(N_i) \times \square_{i\in\IN} G_L(K_i) 
\ar[r]^{\mbox{\hspace{20mm} $\rho$}} \ar[d]_{\mbox{\small $\id \times \lambda_\K$}}^{\mbox{\small \ $\approx$}} & G \\ 
\square_{i\in\IN} G(N_i) \times (\im \lambda_{\K} \cap G_L) 
\ar[r]_{\mbox{\small \hspace{6mm} $\subset$}} & 
\square_{i\in\IN} G(N_i) \times G_L \ar[u]_{\mbox{\,$\theta$}}
}$$ 
Since $\theta (s \times \id) \phi = \id_\V$,
 the map $\theta$ has the next local section at $\id_M$: 
$$\sigma_0 = (s \times \id) \phi : (\V, \V \cap G_c) \to
 (\square, \cbox)_{i\in\IN} G(N_i) \times (G_L, G_{L, c}).$$
Since $\im (\id \times \lambda_\K)$ is an open neighborhood of
 $$\sigma_0(\id_M) = ((\id_M, \id_M, \dots), \id_M)$$ 
 in $\square_{i\in\IN} G(N_i) \times G_L$, 
 we have a neighborhood $\U$ of $\id_M$ in $\V$
 such that $\sigma_0(\U) \subset \im (\id \times \lambda_\K)$. 
The following map is a local section of $\rho$ at $\id_M$: 
$$\sigma = (\id \times \lambda_{\K}^{-1})\, \sigma_0|_{\U} :
 (\U, \U \cap G_c) \to (\square, \cbox)_{i\in\IN} G(N_i)
 \times (\square, \cbox)_{i\in\IN} G_L(K_i).$$
This completes the proof. 
\end{proof}

\begin{proposition}\label{p_loc-homeo}
Suppose $(M_i, L_i, N_i)_{i\in\IN}$ is an exhausting sequence for $M$.
If $(N_{2i}, L_{2i})$ has LSP$_G$ and $L_{2i} \in \F$ for each $i \in \IN$,
 then the following hold:
\begin{itemize}
\item[(1)] $(G, G_c) \approx_\ell
 \big(\square, \cbox)_{i\in\IN} \widehat{\EE}^G(L_{2i}, M)
 \times (\square, \cbox)_{i\in\IN} G(L_{2i-1}),$ 
\vskip 1mm
\item[(2)] The multiplication map $p : \cbox_{i\in\IN} G(M_i) \to G_c$
 has a local section at any point of $G_c$. 
\end{itemize}
\end{proposition}

\begin{proof}
We apply Lemma \ref{l_LSP} to the discrete families 
 $\LL = (L_{2i})_{i \in \IN}$, $\N = (N_{2i})_{i\in\IN}$
 and $\K = (L_{2i-1})_{i\in\IN}$.
Let $L = \bigcup_{i\in\IN} L_{2i}$.
Then $G_{L}(L_{2i-1}) = G(L_{2i-1})$ for each $i \in \IN$
 because $L_{2i-2}$ is regular closed.
The statement (1) is none other than Lemma \ref{l_LSP}\,(1).
It remains to show the statement (2). 
Due to Lemma \ref{l_LSP}\,(2), the map 
$$\begin{array}[c]{l} 
\rho : \cbox_{i\in\IN} G(N_{2i}) \times \cbox_{i\in\IN} G(L_{2i-1})
 \longrightarrow G_c, \\[2mm] 
\hspace*{15mm} \rho((f_i)_{i\in\IN}, (g_i)_{i\in\IN})
 = \lambda_{\N}((f_i)_{i\in\IN})\lambda_{\K}((g_i)_{i\in\IN})
 \end{array}$$ 
 has a local section at $\id_M$, say
$$\sigma : \W \longrightarrow
 \cbox_{i\in\IN} G(N_{2i}) \times \cbox_{i\in\IN} G(L_{2i-1}).$$
Note that the image $\sigma(h) = ((f_i)_{i\in\IN}, (g_i)_{i\in\IN})$
 of each $h \in \W$ satisfies the following conditions:
\begin{itemize}
\item[(a)] $h = \lambda_{\N}((f_i)_{i\in\IN})\lambda_{\K}((g_i)_{i\in\IN}) 
=  (f_1f_2\cdots)(g_1g_2\cdots) = f_1g_1f_2g_2f_3g_3 \cdots;$ 
\vskip 1mm 
\item[(b)] 
$f_i \in G(N_{2i}) \subset G(M_{2i+1})$ \ and \ 
$g_i \in G(L_{2i-1}) \subset G(M_{2i-1}) 
\subset G(M_{2i+2})$ \ \ for each $i \in \IN$;
\vskip 1mm
\item[(c)]
 $\big(\id_M, \id_M, f_1, g_1, f_2, g_2, \dots\big) \in \cbox_{i\in\IN} G(M_i)$ \ and \ 
 $h = p\big(\id_M, \id_M, f_1, g_1, f_2, g_2, \dots\big)$.
\end{itemize}
Therefore, a local section at $\id_M$ of
 the map $p : \cbox_{i\in\IN} G(M_i) \to G_c$ is defined by 
\[ s : \W \to \cbox_{i\in\IN} G(M_i),
\quad s(h) = \big(\id_M, \id_M, f_1, g_1, f_2, g_2, \dots\big). \]
The conclusion now follows from Lemma~\ref{lem_loc-sec_exh}\,(1). 
This completes the proof. 
\end{proof}

\subsection{Strong topology with respect to an admissible collection} 
\label{Trans-gr-strong-top-2} \mbox{}

In the cases
 where $G$ is the diffeomorphism group of a smooth $n$-manifold ($n \geq 1$)
 or the homeomorphism group of a topological $n$-manifold ($n = 1, 2$),
 we can apply Proposition~\ref{p_loc-homeo}. 
However, when $G$ is the homeomorphism group of
 an $n$-manifold $M$ ($n \geq 3$),
 it is still an open problem whether
 the restriction map \break 
 $r : G \to \EE^G(L, M)$ has a local section
 for a locally flat $n$-submanifold $L$ of $M$
 (cf.\ Subsection \ref{Homeo-gr}). 
At this moment,
 we only know that the deformation theorem for embeddings
 in topological manifolds (\cite{EK}) implies the weak local section property. 
This motivates the formulation of this section. 

Throughout this subsection,
 we assume that $G$ is a transformation group on a space $M$. 

\begin{definition}\label{d_adm-coll}
An {\em admissible collection} $(\F, \atau{}{\ast}{M})$
 for the transformation group $G$ on $M$ 
 consists of a collection $\F$ of subsets of $M$ and 
 an assignment of an admissible topology $\atau{}{L}{M}$
 on the set $\EE^G(L, M)$ to each $L \in \F$. 
\end{definition}

\begin{convention}
When an admissible collection $(\F,\atau{}{\ast}{M})$ is fixed,
 for any triple $(N, L, K)$ of subsets of $M$ with $L \in \F$, 
 let $\atau{K}{L}{N}$ denote the subspace topology on $\EE^G_{K}(L,N)$
 inherited from the space $(\EE^G_{}(L, M), \atau{}{L}{M})$,
 which is an admissible topology on $\EE^G_{K}(L, N)$. 
\end{convention}

\begin{definition} \label{d_strong-top-2}
Let $(\F, \atau{}{\ast}{M})$ be an admissible collection
 for the transformation group $G$ on $M$. 
We say that $G$ has {\em a strong topology}
 with respect to $(\F, \atau{}{\ast}{M})$
 if $G$ has a strong topology and
 satisfies the following additional condition:
\begin{itemize}
\item[$(\ast \ast)$]
For any discrete sequence $\LL = (L_i)_{i\in\IN}$ in $M$ with $L_i \in \F$,
 the following function is continuous:
$$r_\LL : G \to \square_{i\in\IN} (\EE^G(L_i, M), \atau{}{L_i}{M}).$$
\end{itemize}
\end{definition}

\begin{assumption} 
Below we assume that $M$ is a locally compact $\sigma$-compact space and
 $G$ is a transformation group on $M$ with a strong topology
 with respect to an admissible collection $(\F, \atau{}{\ast}{M})$. 
\end{assumption}

Suppose $\LL = (L_i)_{i\in\IN}$, $\N = (N_i)_{i\in\IN}$ and
 $\K = (K_i)_{i\in\IN}$ are discrete sequences of subsets of $M$ 
 such that $L_i \subset N_i$ $(i \in \IN)$ and 
 $M = L \cup K$. 
Here, $L = \bigcup_{i\in\IN} L_i$, $N = \bigcup_{i\in\IN} N_i$
 and $K = \bigcup_{i\in\IN} K_i$. 
Since $G$ has a strong topology, 
 the sequences $\N$ and $\K$ induce the open embeddings 
$$\lambda_\N : \square_{i\in\IN} G(N_i) \to G(N) \quad \text{and} \quad 
 \lambda_\K: \square_{i\in\IN} G(K_i) \to G(K).$$ 
These maps determine the map 
$$\rho : (\square, \cbox)_{i\in\IN}G(N_i)
 \times (\square, \cbox)_{i\in\IN}G_L(K_i) \to (G, G_c), \hspace{4mm} 
\rho((g_i)_{i\in\IN},(h_i)_{i\in\IN})
 = \lambda_\N((g_i)_{i\in\IN})\lambda_\K((h_i)_{i\in\IN}).$$

Under the weaker condition WLSP$_G$, we obtain the following conclusions. 

\begin{lemma}\label{lem_WLSP} 
If $(N_i, L_i)$ has WLSP$_G$ and $N_i \in \F$ for each $i \in \IN$,
 then the map $\rho$ has a local section at $\id_M$. 
\end{lemma}

\begin{proof}
The proof is exactly a repitition of the arguments in 
Lemma~\ref{l_LSP}
 (except the part (1)). 
There is only one point to be modified: 
\begin{itemize}
\item[$(\dagger)$] The map $r_\LL$ is replaced by the map $r_\N$
 so that $\V_i$ is an open neighborhood of the inclusion 
\mbox{$i_{N_i} : N_i \subset M$} in the space 
 $(\EE^G(N_i, M),\atau{}{N_i}{M})$.
\end{itemize} 
The remaining parts are unchanged. 
\end{proof}

\begin{proposition}\label{prop_loc-sec}
Suppose $(M_i, L_i, N_i)_{i\in\IN}$ is an exhausting sequence for $M$.
If each $(N_{2i}, L_{2i})$ has WLSP$_G$ and $N_{2i} \in \F$,
 then the multiplication map $p : \cbox_{i\in\IN} G(M_i) \to G_c$
 has a local section at any point of $G_c$.
\end{proposition}

\begin{proof}
The proof is completely same as that of Proposition~\ref{p_loc-homeo}\,(2), 
 except that we apply Lemma~\ref{lem_WLSP} instead of Lemma~\ref{l_LSP}\,(2). 
\end{proof}


\section{Homeomorphism and diffeomorphism groups of non-compact manifolds}
\label{Homeo-Diffeo}

In this section,
 we apply Propositions~\ref{p_loc-homeo} and \ref{prop_loc-sec}
 in Section \ref{Trans-gr-strong-top} to study local topological properties of
 homeomorphism groups and diffeomorphism groups of non-compact manifolds $M$
 endowed with the Whitney topology.
For homeomorphism groups of topological $n$-manifolds $M$, 
 in case $n = 2$ we can describe the local topological type
 of $\HH(M)$ and $\HH_c(M)$,
 while in case $n \geq 3$ the Homeomorphism Group Problem is still open and 
 we are restricted only to show the local contractibility of $\HH_c(M)$. 
On the other hand, for diffeomorphism groups of smooth $n$-manifolds,
 we can determine the local topological type of $\DD(M)$ and $\DD_c(M)$
 in every dimension $n$.

\subsection{Homeomorphism groups of non-compact $n$-manifolds}
\label{Homeo-gr} \mbox{}

Suppose $M$ is a $\sigma$-compact topological $n$-manifold
 possibly with boundary.
When $M$ is compact,
 the group $\HH(M)$ is known to be locally contractible
 (\cite{Cer} and \cite{EK}). 
In this subsection,
 we apply Proposition \ref{prop_loc-sec} to extend this result
 to the noncompact case.

\begin{proposition}\label{p_Homeo_LC}
For every $\sigma$-compact $n$-manifold $M$ possibly with boundary,
 the group $\HH_c(M)$ is locally contractible.
\end{proposition}

We follow the formulation in Subsection~\ref{Trans-gr-strong-top-2}. 
The topological group $G = \HH(M)$ admits the natural action on $M$.
According to the convension of Subsection \ref{Gen-trans-gr},
 we use the following notations for $K, N \subset M$: 
\begin{gather*}
\HH_0(M) = G_0, \quad \HH_c(M) = G_c, \quad \HH(M, K) = G_K, \\
 \quad \HH(M, M \setminus N) = G(N), 
 \quad \HH(M, K \cup (M \setminus N)) = G_K(N), \\ 
 \quad \HH_c(M, K) = G_{K, c},
 \quad \HH_0(M, K) = (G_K)_0.
\end{gather*}
\noindent
For subspaces $K \subset L \subset N \subset M$, 
 the symbol $\EE_K(L,N)$ denotes the space of embeddings 
 $f : L \to N$ with $f|_K = \id_K$ endowed with the compact-open topology
 (Section \ref{Basic}). 
Recall that an embedding $f:L\to M$ is {\em proper}
 if $f^{-1}(\partial M)=L\cap\partial M$. 
Let $\EE^*_K(L,M)$ denote the subspace of $\mathcal E_K(L,M)$
 consisting of all proper embeddings.
Then $\EE^G_K(L,M) \subset \EE^*_K(L,M)$.

Let $\F$ denote the collection of all subsets of $M$ and 
 $\EE^G(L, M)$, $L \in \F$, be endowed
 with the compact-open topology $\atau{}{L}{M}$,
 which is admissible, that is,
 the action of $G = \HH(M)$ on $\EE^G(L, M)$ is continuous.
Then the transformation group $G$ on $M$ has a strong topology
 with respect to the admissible collection $(\F, \atau{}{\ast}{M})$,
 that is, for each discrete sequence $\LL = (L_i)_{i\in\IN}$ in $M$ 
 (i) the map $\lambda_{\LL} : \square_{i\in\IN}G(L_i) \to G(\bigcup_{i\in\IN}L_i)$
 is an open embedding and (ii) 
 the map $r_{\LL} : G \to \square_{i\in\IN}\EE^G(L_i,M)$ is continuous
 with respect to the compact-open topology $\atau{}{L_i}{M}$. 
Next, we shall inspect WLSP$_G$ of compact subsets of $M$. 

\begin{lemma}\label{lem_WLSP_Homeo} 
Suppose $L$ is a compact subset of $M$ and
 $N$ is a compact neighborhood of $L$ in $M$.
Then the pair $(N, L)$ has the WLSP$_G$
 (with respect to the topology $\atau{}{N}{M}$),
 that is, there exists an open neighborhood $\V$ of $i_N : N \subset M$ in
 $(\EE^G(N,M),\atau{}{N}{M})$ and a continuous map $s : \V \to G(N)$
 such that $s(f)|_L = f|_L$ for each $f \in \V$.
\end{lemma}

This lemma follows from the next version of the deformation theorem
 for topological embeddings \cite{EK}.  

\begin{lemma} \label{lem_EK}
Suppose $C \subset D$ are compact subsets of $M$ with 
$C \subset \tint_M D$
 and $K \subset L$ are closed subsets of $M$ with $K \subset \tint_M L$.
Then there exists an open neighborhood $\V$ of
 the inclusion 
 $i : D \cup L \subset M$ in $\EE^\ast_L(D \cup L, M)$
 and a map $$\phi : \V \to \HH(M, K \cup (M \setminus \tint_M D))$$ such that
 $\phi(f)|_C = f|_C$ for each $f \in \V$ and $\phi(i) = \id_M$.
\end{lemma}

\begin{proof}
Take a compact neighborhood $E$ of $C$ in $\tint_M D$. 
Then, by \cite[Theorem 5.1]{EK} 
there exists an open neighborhood $\U$ of the inclusion $i_D : D \subset M$ 
in $\EE_{D \cap L}^\ast(D, M)$ 
and a map $\eta : \U \to \EE_{D \cap K}^\ast(D, M)$ such that 
$\eta(i_D) =  i_D$ and for each $f \in \U$ 
$$\mbox{
(a) $\eta(f) =  \id$ on $E$, \quad 
(b) $\eta(f) =  f$ on $\bd_M D$ \quad  and \quad 
(c) $\eta(f)(D) = f(D)$.}$$ 

Replacing $\U$ by a smaller one,
 we may assume that $f(C) \subset E$.
The required map $\phi$ is defined by
 $\phi(f)|_D = \eta(f|_D)^{-1}(f|_D)$ and $\phi(f) = \id$ on $M \setminus D$.
\end{proof}

Let $(M_i, K_i)_{i\in\IN}$ be an exhausting sequence of $M$,
 that is, 
 $M = \bigcup_{i\in\IN}M_i$,
 each $M_i$ is compact regular closed in $M$,
 $M_i \subset \tint_M M_{i+1}$
 and $K_i = M \setminus \tint_M M_i$.
Then $\HH(M,K_i) = G_{K_i} = G(M_i)$ for each $i \in \IN$.
Consider the multiplication map 
$$p : \cbox_{i\in\IN} \HH(M, K_i) \to \HH_c(M),
 \quad p((h_i)_{i\in\IN}) = h_1 h_2 h_3 \cdots.$$

\begin{lemma} \label{l_local_section}
The map $p : \cbox_{i\in\IN} \HH(M, K_i) \to \HH_c(M)$
 has a local section at $\id_M$. 
\end{lemma} 

\begin{proof}
We can form an exhausting sequence $(M_i, K_i, L_i, N_i)_{i\in\IN}$,
 that is, $L_i = M_i \setminus \tint_M M_{i-1}$,
 $L_i \subset \tint_M N_i$ and $N_i \cap N_j = \emptyset$ if $|i - j| \geq 2$. 
Since each pair $(N_i, L_i)$ has WLSP$_G$ by Lemma \ref{lem_WLSP_Homeo},
 the conclusion follows from Proposition~\ref{prop_loc-sec}. 
\end{proof}

\begin{proof}[\bf Proof of Proposition~\ref{p_Homeo_LC}] 
By \cite[Theorem 0]{AP} (cf.\ \cite{KS}, \cite{Quinn}), 
there exists an exhausting sequence $(M_i)_{i \in \IN}$ for $M$ 
consisting of compact $n$-submanifolds of $M$ such that 
 $\bd_M M_i$ is a compact proper $(n-1)$-submanifold of $M$
 which is transversal to $\partial M$. 
Since $$G(M_i) = \HH(M, M \setminus \tint_M M_i)
 \approx \HH(M_i, \bd_M M_i)$$ and 
 the latter is locally contractible by \cite[Corollary 7.3]{EK}, 
 the conclusion follows from
 Lemmas \ref{l_local_section} and \ref{lem_loc-sec_exh}\,(2). 
\end{proof}

\subsection{Homeomorphism groups of non-compact surfaces} \mbox{}

In this subsection, we shall recognize the local topological type of
 the pair $(\HH(M,K), \HH_c(M,K))$ for a $2$-manifold $M$
 and a closed polyhedral subset $K\subset M$. 
Suppose $M$ is a $\sigma$-compact $2$-manifold possibly with boundary.
Then $M$ admits a combinatorial triangulation unique up to
 PL-homeomorphisms \cite{Moise}.
We fix a triangulation of $M$ and regard $M$ as a PL $2$-manifold.
A {\em subpolyhedron} of $M$ means a subpolyhedron
 with respect to this PL-structure. 
A $2$-submanifold of $M$ means 
a subpolyhedron $N$ of $M$ such that
 $N$ is a $2$-manifold and $\bd_M N$ is transverse to $\partial M$
 so that $\bd_M N$ is a proper 1-submanifold of $M$ and 
 $M \setminus \tint_M N$ is also a $2$-manifold.
Let $\HH^{PL}(M)$ denote the subgroup of $\HH(M)$
 consisting of PL-homeomorphisms with respect to the  PL-structure of $M$,
 and set $\HH^{PL}_c(M,K) = \HH^{PL}(M) \cap \HH_c(M,K).$  

Theorem~\ref{thm_main_1} and Proposition~\ref{prop_main_2} in Introduction 
 follow from the next theorem with taking $K = \emptyset$. 

\begin{theorem}\label{thm_surface}
Suppose $M$ is a non-compact $\sigma$-compact $2$-manifold 
possibly with boundary and
 $K\subsetneqq M$ is a subpolyhedron.
\begin{enumerate}
\item If $\cl_M(M\setminus K)$ is compact, then
 {\rm (i)} $\HH(M,K)$ is an $l_2$-manifold and hence
 {\rm (ii)} $\HH_0(M,K)$ is an open normal subgroup of $\HH(M,K)$. 
\item If $\cl_M(M\setminus K)$ is non-compact, then
\begin{itemize}
\item[(i)\ ]
$(\HH(M,K), \HH_c(M,K))$ is locally homeomorphic to 
 $(\square^\w l_2,  \cbox^\w l_2)$, 
 hence $\HH(M,K)$ is locally homeomorphic to $\square^\w l_2$
 and $\HH_c(M,K)$ is an $(l_2 \times \IR^\infty)$-manifold, 
\vskip 1mm 
\item[(ii)\,] $\HH_0(M,K)$ is an open normal subgroup of $\HH_c(M,K)$
 and thus 
$$\hspace*{10mm} \HH_c(M,K) \approx \HH_0(M,K) \times \M_c(M,K),$$ 
where \ $\M_c(M,K) = \HH_c(M,K)/\HH_0(M,K)$ (with the discrete topology).
\end{itemize}
\item The subgroup $\HH^{PL}_c(M,K)$ is homotopy dense in $\HH_c(M,K)$. 
\end{enumerate}
\end{theorem}

We keep the notations for $G = \HH(M)$ listed in Subsection \ref{Homeo-gr}. 
For $K \subset L\subset M$, the symbols 
$$\EE_K(L,M) \supset \EE^*_K(L,M) \supset \EE^\star_K(L,M) = \EE^G_K(L,M)$$
 denote the space of embeddings and the subspaces of proper embeddings
 and extendable embeddings, respectively. 
These spaces are endowed with the compact-open topology. 

To prove Theorem~\ref{thm_surface},
 we use the next two theorems besides Proposition~\ref{p_loc-homeo}. 

\begin{theorem}\label{thm_LM} {\rm(\cite{GH}, \cite{LM}, cf.\ \cite{Y2})} 
Suppose $M$ is a compact 
$2$-manifold possibly with boundary
 and $K \subsetneqq M$ is a subpolyhedron.
Then, {\rm (i)} $\HH(M,K)$ is an $l_2$-manifold
 and {\rm (ii)} $\HH^{PL}(M,K)$ is homotopy dense in $\HH(M,K)$. 
\end{theorem}

The following is a slight extension of the results in \cite{Yag} and \cite{Y2}.
Note that the assertion was verified in \cite{LM}
 in the most important case that $K = \emptyset$ and
 $L$ is either a proper arc, an orientation-preserving circle or a compact 2-submanifold of $M$. 

\begin{theorem}\label{thm_bdl_surface} 
Suppose $M$ is a $2$-manifold possibly with boundary
 and $K \subset L$ are two subpolyhedra of $M$ 
 such that $\cl_M(L \setminus K)$ is compact.
\begin{enumerate}
\item 
For every closed subset $C$ of $M$
 with $C \cap \cl_M(L \setminus K) = \emptyset$,
 the restriction map
$$r :\HH(M,K) \to \mathcal E^*_K(L,M),
 \quad r(h) = h|_L$$
 has a local section $s : \U \to \HH_0(M,K\cup C) \subset \HH(M,K)$
 at the inclusion $i_L : L \subset M$.
\item
The restriction map $r :\HH(M,K) \to \mathcal E^\star_K(L,M)$
 is a principal $\HH(M,L)$-bundle.
\item 
\begin{itemize}
\item [(i)\,] $\mathcal E^\star_K(L,M)$ is an open neighborhood
 of the inclusion $i_L$ in $\mathcal E^*_K(L,M)$. 
\item[(ii)] The spaces $\mathcal E^*_K(L,M)$ and $\mathcal E^\star_K(L,M)$ 
 are $l_2$-manifolds if $\dim(L\setminus K) \ge 1$.
\end{itemize}
\end{enumerate}
\end{theorem} 

\begin{proof}
In \cite[Proposition 4.2]{Yag} and \cite[Theorem 2.1]{Y2}, 
 we verified the case where $L$ is compact. 
The general case is obtained
 if we choose a regular neighborhood $N$ of $\cl_M(L \setminus K)$
 in $M \setminus C$ and apply the compact case to the data 
 $$(N, (L \cap N) \cup \bd_M N, (K \cap N) \cup \bd_M N).$$
\vskip -6.5mm
\end{proof}

\begin{proof}[\bf Proof of Theorem~\ref{thm_surface}] 
Consider the subgroup $G_K = \HH(M, K)$ of $G = \HH(M)$ and 
 the collection $\F$ of all compact subpolyhedra of $M$. 
Then the transformation group $G_K$ on $M$ has a strong topology
 with respect to $\F$. 
Choose any exhausting sequence $(M_i, L_i, N_i)_{i\in\IN}$ for $M$
 such that each $M_i$ and $N_i$ are compact $2$-submanifolds of $M$. 
Recall that $L_i = M_i \setminus \tint_M M_{i-1} \subset \tint_M N_i$.
Then, for each $i \in \IN$ 
\begin{itemize}
\item[(i)\ ] $L_{2i} \in \F$, \quad
 (ii) $(N_{2i}, L_{2i})$ has LSP$_{G_K}$ \quad and 
\item[(iii)] the quotient topology on the space 
$$\hspace*{10mm} \EE^{G_K}(L_{2i}, M) \cong \EE^G_K(K \cup L_{2i}, M)
 = \EE^\star_K(K \cup L_{2i}, M)$$   
 coincides with the compact-open topology. 
\end{itemize}
The assertion (ii) is verified by applying Theorem~\ref{thm_bdl_surface}\,(1) 
 to the data $(L, K, C) = (K \cup L_{2i}, K, M \setminus \mint\,N_{2i})$ and 
 (iii) follows from Remark~\ref{rem_LSP_G}\,(0). 
Hence, it follows from Proposition~\ref{p_loc-homeo} that  
\begin{itemize}
\item[$(\ast)_1$] $(G_K, G_{K,c}) \approx_\ell
 \big(\square,\cbox)_{i\in\IN} \EE^{G_K}(L_{2i},M)
  \times (\square, \cbox)_{i\in\IN} G_K(L_{2i-1})$,  
\item[$(\ast)_2$]  the multiplication map
 $p : \cbox_{i\in\IN} G_K(M_i) \to G_{K, c}$ 
 has a local section at any point of $G_{K, c}$. 
\end{itemize}

(1)
Take a compact submanifold $N$ of $M$
 such that $\cl_M(M \setminus K) \subset N$.
Then $\HH(M,K)$ is identified with $\HH(N,(N \cap K) \cup \bd_M N)$ and 
 the assertion (1) is the direct consequence of Theorem~\ref{thm_LM}. 
\smallskip 

(2) Since 
$G_K(L_{2i-1}) = \HH(M, K \cup (M \setminus \tint_M L_{2i-1})  
\approx \HH(L_{2i-1}, (K \cap L_{2i-1}) \cup \bd_M L_{2i-1})$ 
and $L_i \not\subset K$ for infinitely many $i\in\IN$, 
 by Theorem~\ref{thm_LM}\,(i) and Theorem~\ref{thm_bdl_surface}\,(3)(ii) 
 we have
$$\big(\square, \cbox)_{i\in\IN} \EE^{G_K}(L_{2i}, M)
 \times (\square, \cbox)_{i\in\IN} G_K(L_{2i-1})
 \approx_\ell (\square^\w l_2,\cbox^\w l_2).$$ 
Hence, the assertion (2) follows from $(\ast)_1$.

(3) Let $H = \HH^{PL}(M)$ and consider the subgroup $H_K$ of $G_K$.  
Since $G_c$ is paracompact, so is $G_{K, c}$. 
Since 
\begin{align*}
(G_K(M_i), H_K(M_i)) &= (\HH(M, K \cup K_i), \HH^{PL}(M, K \cup K_i)) \\
& \approx 
(\HH(M_i, (M_i \cap K) \cup \bd_M M_i),
 \HH^{PL}(M_i, (M_i \cap K) \cup \bd_M M_i)), 
\end{align*}
 Theorem~\ref{thm_LM}\,(ii) implies that $H_K(M_i)$ is HD in $G_K(M_i)$. 
By $(\ast)_2$,
 we can apply Lemma~\ref{lem_loc-sec_exh}\,(3)
 to assert that $\HH^{PL}_c(M,K) = H_{K, c}$ is HD in $\HH_c(M,K) = G_{K, c}$.
\end{proof}


\subsection{Diffeomorphism groups of non-compact smooth manifolds} \mbox{}

In this subsection,
 we study diffeomorphism groups of non-compact smooth manifolds
 endowed with the Whitney $C^\infty$-topology.
Suppose $M$ is a smooth $\sigma$-compact $n$-manifold without boundary.
Let $\DD(M)$ denote the group of diffeomorphisms of $M$ endowed with
 the Whitney $C^\infty$-topology
 (= the very-strong $C^\infty$-topology in \cite{Illman}).

The topological group $G = \DD(M)$ admits the natural action on $M$.
Similarly to the previous subsection,
 we use the following notations for $K, N \subset M$:
\begin{gather*}
\DD_0(M) = G_0, \ \ 
\DD_c(M) = G_c, \ \ 
\DD(M, K) = G_K, \\ 
\DD(M, M \setminus N) = G(N), \ \ 
\DD(M, K \cup (M \setminus N)) = G_K(N), \\ 
\DD_c(M, K) = G_{K,c}, \ \
\DD_0(M, K) = (G_K)_0.
\end{gather*}
Since the inclusion map $\DD(M) \subset \HH(M)$ is continuous,
 it follows from Proposition~\ref{prop_tr-gr_compsupp}
 that $\DD_0(M) \subset \DD_c(M)$. 
The quotient group $\M_c^\infty(M) = \DD_c(M)/\DD_0(M)$
 (with the quotient topology) is called the mapping class group of $M$.
 
Let $\F^\infty_c$ denote the collection of all compact smooth $n$-submani\-folds
 of $M$.
For $L \in \F^\infty_c$ and a subset $K \subset L$,
 let $\EE_K^\infty(L, M)$ denote the space of $C^\infty$-embeddings
 $f : L \to M$ with $f|_K = \id_K$ and
 let $\EE_K^{\infty, \star}(L, M) = \EE_K^G(L, M)$
 (the space of extendable $C^\infty$-embeddings).
These spaces are endowed with the compact-open $C^\infty$-topology 
 (and its subspace topology).
There is a natural restriction map
$$r : \DD(M, K) \to \EE^\infty_K(L, M), \quad r(h) = h|_L.$$

The following is the main result of this subsection.

\begin{theorem}\label{t_diff} 
Suppose $M$ is a non-compact $\sigma$-compact smooth $n$-manifold
 without boundary.
 \begin{enumerate}
\item $(\DD(M), \DD_c(M)) \approx_\ell (\square^\w l_2, \cbox^\w l_2)$.
Hence,
$\DD(M) \approx_\ell \square^\w l_2$ and
 $\DD_c(M)$ is an $( l_2 \times \IR^\infty)$-manifold.
\item
$\DD_0(M)$ is an open normal subgroup of $\DD_c(M)$ and 
 $$\DD_c(M) \approx \DD_0(M) \times \M_c^\infty(M).$$ 
\end{enumerate}
\end{theorem}

In the proof,
 we use Proposition~\ref{p_loc-homeo} and the following bundle theorem
 (cf.\ \cite{Cerf}, \cite{Hmlt}, \cite{Les}, \cite{Palais1}, \cite{Seeley}). 

\begin{theorem}\label{t_bdl} 
Suppose $K$ and $L$ are smooth $n$-submanifolds of $M$ 
such that they are closed subsets of $M$, $K \subset \tint_M L$ and
 $\cl_M(L \setminus K)$ is compact and nonempty.
\begin{enumerate}
\item 
For any closed subset $C$ of $M$ with $C \cap L = \emptyset$, 
 the restriction map $r : \DD(M, K) \to \EE_K^\infty(L, M)$
 has a local section
$$s : \U \to \DD(M, K \cup C) \subset \DD(M, K)$$
 at the inclusion $i_L : L \subset M$ such that $s(i_L) = \id_M$.
\item
The spaces $\DD(M, K \cup (M \setminus L))$ and $\EE^\infty_K(L, M)$
 are infinite-dimensional separable Fr\'echet manifolds
 $($thus topological $l_2$-manifolds$)$ and $\EE^{\infty, \star}_K(L, M)$
 is an open subset of $\EE^\infty_K(L, M)$.
\end{enumerate}
\end{theorem}

\begin{proof}[\bf Proof of Theorem \ref{t_diff}]
We follow the formulation of Subsections 5.1--5.2 and
 apply Proposition~\ref{p_loc-homeo} to
 the transformation group 
 $G = \DD(M)$ on $M$. 
For $L \in \F^\infty_c$,  
the compact-open $C^\infty$-topology on $\EE^G(L, M)$ 
is an admissible topology, and 
Theorem~\ref{t_bdl}\,(1) and Remark~\ref{rem_adm-top}\,(iii) imply that
 this topology coincides with the quotient topology $\ahtau{}{L}{M}$
 on $\EE^G(L, M)$ induced by the restriction map $r : G \to \EE^G(L, M)$. 
Hence, it is seen that the transformation group $G$
 has a strong topology with respect to $\F^\infty_c$. 
 
Now we can apply Proposition~\ref{p_loc-homeo}
to any exhausting sequence $(M_i, L_i, N_i)_{i\in\IN}$ for $M$
 such that each $M_i$ is a compact $n$-submanifold of $M$. 
For each $i \in \IN$, it is seen that $L_{2i} \in \F^\infty_c$
 and $(N_{2i}, L_{2i})$ has LSP$_G$ by Theorem~\ref{t_bdl}\,(1). 
Hence, from Proposition~\ref{p_loc-homeo} it follows that 
$$(G, G_c) \approx_\ell \big(\square, \cbox)_{i\in\IN} \EE^G(L_{2i}, M)
 \times (\square, \cbox)_{i\in\IN} G(L_{2i-1}).$$ 
The latter pair is locally homeomorphic to $(\square^\w l_2,\cbox^\w l_2)$
 by Theorem~\ref{t_bdl}\,(2). 

Due to Theorem~\ref{t_bdl}\,(2), 
 $G(L)$ is separable metrizable for each compact smooth
 $n$-submanifold $L \subset M$. 
Then $G_c$ is paracompact by Proposition \ref{prop_tr-gr_compsupp}\,(2).
\end{proof}


\begin{thebibliography}{99}

\bibitem{AP}
S.\,Alpern and V.\,Prasad,
\textit{End behaviour and ergodicity for homeomorphisms of manifolds
 with finitely many ends},
 Canad. J. Math. \textbf{39} (2) (1987),  473--491.

\bibitem{An}
 R.\,D.\,Anderson,
\textit{Spaces of homeomorphisms of finite graphs},
 unpublished preprint.

\bibitem{AB}
 R.\,D.\,Anderson and R.\,H.\,Bing,
\textit{A complete elementary proof that Hilbert space
 is homeomorphic to the countable infinite product of lines,}
  Bull. Amer. Math. Soc. \textbf{74} (1968),  771--792.

\bibitem{Ban}
 T.\,Banakh,
\textit{On hyperspaces and homeomorphism groups homeomorphic to
 products of absorbing sets and $\IR^\infty$},
 Tsukuba J. Math. \textbf{23} (1999),  495--504.

\bibitem{BMS}
 T.\,Banakh, K.\,Mine and K.\,Sakai,
\textit{Classifying homeomorphism groups of infinite graphs}, Topology Appl. \textbf{157} (2009), 108--122.

\bibitem{BMSY} 
T.\,Banakh, K.\,Mine, K.\,Sakai and T.\,Yagasaki, 
\textit{Topological groups locally homeomorphic to LF-spaces},
 preprint. 

\bibitem{BR}
T.\,Banakh and D.\,Repo\v s, {\em A topological characterization of LF-spaces}, preprint (arXiv:0911.0609).

\bibitem{BY}
T.\,Banakh and T.\,Yagasaki, 
\textit{The diffeomorphism groups of the real line are pairwise bihomeomorphic}, 
in: Proceedings of ``Infinite-Dimensional Analysis and Topology'', Yaremche, Ivano-Frankivsk, Ukraine, 2009; a special issue in Topology, in press, online November 2009 (arXiv:0804.3645v3).

\bibitem{Cerf}
 J.\,Cerf,
\textit{Topologie de certains espaces de plongements},
 Bull. Soc. Math. France \textbf{89} (1961), 227--380.

\bibitem{Cer}
 A.\,V.\,{\v C}ernavski\u\i,
\textit{Local contractibility of the group of homeomorphisms of a manifold},
 (Russian) Mat. Sb. (N.S.) \textbf{79 (121)} (1969), 307--356.

\bibitem{EK}
 R.\,D.\,Edward and R.\,C.\,Kirby,
\textit{Deformations of spaces imbeddings},
 Ann. of Math. \textbf{93} (1971), 63--88.

\bibitem{En}
 R.\,Engelking,
 \textit{General Topology},
 Revised and completed edition, Sigma Ser. in Pure Math. \textbf{6},
 Heldermann Verlag, Berlin, 1989.

\bibitem{Gau}
D.\,B.\,Gauld,
\textit{The graph topology for function spaces},
Indian J. Math. \textbf{18} (1976), 125--132.

\bibitem{Geo}
 R.\,Geoghegan,
\textit{On spaces of homeomorphisms, embeddings and functions -- I},
 Topology \textbf{11} (1972), 159--177.

\bibitem{GH}
R.\,Geoghegan and W. E. Haver,
\textit{On the space of piecewise linear homeomorphisms of a manifold},
 Proc. Amer. Math. Soc. \textbf{55} (1976), 145--151.

\bibitem{Gru}
G.\,Gruenhage, 
\textit{Generalized metric spaces},
 In: K.~Kunen and J.~Vaughan (eds.),
 Handbook of Set-Theoretic Topology,
 North-Holland, Amsterdam, 1984, 423--501.

\bibitem{GZ}
 I.\,I.\,Guran and M.\,M.\,Zarichnyi,
\textit{The Whitney topology and box products},
 Dokl. Akad. Nauk Ukrain. SSR Ser. A. 1984, no.11, 5--7.

\bibitem{Hmlt}
 R.\,S.\,Hamilton,
\textit{The inverse function theorem of Nash and Moser},
 Bull. Amer. Math. Soc. (N.S.) \textbf{7}:1 (1982), 65--222.

\bibitem{Illman}
 S.\,Illman, 
\textit{The very-strong $C^\infty$ topology on $C^\infty(M,N)$ and $K$-equivariant maps},
 Osaka J. Math. \textbf{40}, no. 2 (2003),  409--428.

\bibitem{KM}
 A.\,Kriegl and P.\,W.\,Michor,
\textit{The Convenient Setting of Global Analysis},
 Math. Surveys and Monog. \textbf{53},
 Amer. Math. Soc., Providence, R.I., 1997.

\bibitem{KS}
R.\,C.\,Kirby and L.\,C.\,Siebenmann,
\textit{Foundational Essays on Topological Manifolds, Smoothings, and Triangulations}, 
Annals of Math. Studies \textbf{88}, 
Princeton University Press, Princeton, N.J., 1977.

\bibitem{Les}
 J.\,A.\,Leslie,
\textit{On a differential structure for the group of diffeomorphisms},
 Topology \textbf{6} (1967), 263--271.

\bibitem{LM}
 R.\,Luke and W.\,K.\,Mason,
\textit{The space of homeomorphisms on a compact two-manifold
 is an absolute neighborhood retract},
 Trans. Amer. Math. Soc. \textbf{164} (1972), 275--285.

\bibitem{Man}
 P.\,Mankiewicz,
 {\it On topological, Lipschitz, and uniform classification of LF-spaces},
 Studia Math. \textbf{52} (1974), 109--142.

\bibitem{Mi}
E.\,Michael, \textit{$\aleph_0$-spaces},
J. Math. Mech. \textbf{15} (1966), 983--1002.

\bibitem{Moise}
 E.\,E.\,Moise,
 \textit{Geometric Topology in Dimensions $2$ and $3$},
 GTM \textbf{47}, Springer-Verlag, New York-Heidelberg, 1977.

\bibitem{Palais1}
 R.\,S.\,Palais,
\textit{Local triviality of the restriction map for embeddings},
 Comment. Math. Helv. \textbf{34} (1960), 305--312.

\bibitem{Quinn}
 F.\,Quinn,
\textit{Ends of maps III, Dimensions 4 and 5},
 J. Differential Geom. \textbf{17}(3) (1982), 503--521.

\bibitem{Seeley}
 R.\,T.\,Seeley,
\textit{Extension of $C\sp{\infty }$ functions defined in a half space},
 Proc. Amer. Math. Soc. \textbf{15} (1964), 625--626.

\bibitem{T1}
 H.\,Toru{\'n}czyk,
\textit{Absolute retracts as factors of normed linear spaces},
 Fund. Math. \textbf{86} (1974),  53--67.

\bibitem{Yag}
 T.\,Yagasaki, 
\textit{Spaces of embeddings of compact polyhedra into 2-manifolds},
 Topology Appl. \textbf{108} (2000), 107--122.

\bibitem{Y2}
 T.\,Yagasaki, 
\textit{Homotopy types of homeomorphism groups of noncompact 2-manifolds},
 Topology Appl. \textbf{108} (2000), 123--136.

\bibitem{W-Prob}
 J.\,E.\,West,
\textit{Open problems in infinite-dimensional topology},
 in: J.\,van Mill and G.\,M.\,Reed, (eds.), Open Problems in Topology,
 Elsevier Sci. Publ. B.V., Amsterdam, 1990, 523--597.

\bibitem{Wil1}
 S.\,Williams,
\textit{Box products},
 In: K.\,Kunen and J.\,Vaughan (eds.),
 Handbook of Set-Theoretic Topology,
 North-Holland, Amsterdam, 1984, 169--200.

\end{thebibliography}
\end{document}